\documentclass[11pt,leqno]{article}
\usepackage{amsthm,amsfonts,amssymb,amsmath,oldgerm}
\usepackage{epsfig}
\numberwithin{equation}{section}
\usepackage[thinlines]{easybmat}

\def\g{\gamma}

\def\eps{\varepsilon }

\newcommand\R{\mathbb R}

\def\g{\gamma}

\def\eps{\varepsilon}

\newcommand\br{\begin{remark}}
\newcommand\er{\end{remark}}
\newcommand\bp{\begin{pmatrix}}
\newcommand\ep{\end{pmatrix}}
\newcommand\be{\begin{equation}}
\newcommand\ee{\end{equation}}
\newcommand\ba{\begin{equation}\begin{aligned}}
\newcommand\ea{\end{aligned}\end{equation}}

\newcommand{\bap}{\begin{app}}
\newcommand{\eap}{\end{app}}
\newcommand{\begs}{\begin{exams}}
\newcommand{\eegs}{\end{exams}}
\newcommand{\beg}{\begin{example}}
\newcommand{\eeg}{\end{exaplem}}
\newcommand{\bpr}{\begin{proposition}}
\newcommand{\epr}{\end{proposition}}
\newcommand{\bt}{\begin{theorem}}
\newcommand{\et}{\end{theorem}}
\newcommand{\bc}{\begin{corollary}}
\newcommand{\ec}{\end{corollary}}
\newcommand{\bl}{\begin{lemma}}
\newcommand{\el}{\end{lemma}}
\newcommand{\bd}{\begin{definition}}
\newcommand{\ed}{\end{definition}}
\newcommand{\brs}{\begin{remarks}}
\newcommand{\ers}{\end{remarks}}

\newtheorem{theo}{Theorem}[section]

\newtheorem{exams}[theo]{Examples}

\numberwithin{equation}{section}

\newcommand{\sgn}{\text{\rm sgn}}
\newtheorem{theorem}{Theorem}[section]
\newtheorem{proposition}[theorem]{Proposition}
\newtheorem{corollary}[theorem]{Corollary}
\newtheorem{lemma}[theorem]{Lemma}
\newtheorem{definition}[theorem]{Definition}

\newtheorem{example}[theorem]{Example}
\newtheorem{remark}[theorem]{Remark}

\newcommand{\reff}[1]{\ref{#1}}

\newcommand{\RM}{\mathbb{R}}

\title{
Metastability of solitary roll wave solutions of the St. Venant
equations with viscosity}

\author{\sc \small
Blake Barker\thanks{Indiana University, Bloomington, IN 47405;
bhbarker@indiana.edu: Research of B.B. was partially supported
under NSF grants no. DMS-0300487 and DMS-0801745.}
~~~
Mathew A. Johnson\thanks{Indiana University, Bloomington, IN 47405;
matjohn@indiana.edu: Research of M.J. was partially supported by an NSF Postdoctoral Fellowship under NSF grant DMS-0902192.}
~~~
L.Miguel Rodrigues\thanks{Universit\'e de Lyon, Universit\'e Lyon 1,
Institut Camille Jordan, UMR CNRS 5208, 43 bd du 11 novembre 1918,
F - 69622 Villeurbanne Cedex, France; rodrigues@math.univ-lyon1.fr: Stay of M.R. in Bloomington was supported by 
Frency ANR project no. ANR-09-JCJC-0103-01.}
~~~
Kevin Zumbrun\thanks{Indiana University, Bloomington, IN 47405;
kzumbrun@indiana.edu:
Research of K.Z. was partially supported
under NSF grants no. DMS-0300487 and DMS-0801745.}
}
\begin{document}

\maketitle


\begin{center}
\end{center}

\begin{center}
{\bf 2000 MR Subject Classification}: 35B35.
\end{center}


\begin{abstract}
We study by a combination of numerical and analytical
Evans function techniques the stability of solitary wave solutions
of the St. Venant equations for viscous shallow-water flow down an
incline, and related models.
Our main result is to exhibit examples of
metastable solitary waves for the St. Venant equations,
with stable point spectrum indicating coherence of the wave
profile but unstable essential spectrum indicating
oscillatory convective instabilities shed in its wake.
We propose a mechanism based on ``dynamic spectrum'' of the
wave profile, by which a wave train of solitary pulses can
stabilize each other by de-amplification of convective instabilities
as they pass through successive waves.
We present numerical time evolution studies supporting these conclusions,
which bear also on the possibility of stable periodic solutions close
to the homoclinic.
For the closely related viscous Jin-Xin model, by contrast,
for which the essential spectrum is stable,
we show using the stability index of Gardner--Zumbrun
that solitary wave pulses are always exponentially unstable, possessing
point spectra with positive real part.
\end{abstract}

\newpage
\tableofcontents

\bigbreak

\section{Introduction }\label{intro}

Roll waves are a well-known phenomenon
occurring in shallow water flow down an inclined ramp,
generated by competition between gravitational force and friction
along the bottom.  Such patterns have been used to model phenomena in
several areas of the engineering literature, including landslides, river and spillway flow, and the topography of
sand dunes and sea beds and their stability properties have been much studied
numerically, experimentally, and by formal asymptotics; see \cite{BM} and references therein.
Mathematically, these can be modeled as traveling-wave solutions
of the St. Venant equations for shallow water flow;
see \cite{D,N1,N2} for discussions of existence in
the inviscid and viscous case.
They may take the form either of periodic wave-trains
or, in the long-wavelength limit of such wavetrains,
of {\it solitary pulse-type waves}.

Stability of periodic roll waves has been considered recently
in \cite{N1,N2,JZN}.
Here, we consider stability of the limiting solitary wave solutions.
This appears to be of interest not only for its relation to stability of
nearby periodic waves (see \cite{GZ,OZ1} and discussion in \cite{JZN}),
but also in its own right.
For, persistent solitary waves are a characteristic feature of shallow
water flow, but have been more typically modeled by dispersive equations
such as Boussinesq or KdV.

Concerning solitary waves of general second-order hyperbolic-parabolic conservation or balance laws such
as those examined in this paper,
all results up to now \cite{AMPZ1,GZ,Z2} have indicated the such solutions exhibit unstable
point spectrum.  Thus, it is not at all clear that an example with stable point spectrum should exist.
Remarkably, however, for the equations considered in \cite{N2,JZN} we
obtain examples of solitary wave solutions that are {\it metastable}
in the sense that they have stable point spectrum, but unstable essential
spectrum of a type corresponding to convective instability.
That is, the perturbed solitary wave propagates relatively
undisturbed down the ramp, while shedding oscillatory instabilities in its wake,
so long as the solution remains bounded by an arbitrary constant.
The shed instabilities appear to grow exponentially forming the typical time-exponential oscillatory
Gaussian wave packets associated with essential instabilities; see the
stationary phase analysis in \cite{OZ1}.  Presumably these would ultimately
blow up according to a Ricatti equation in the standard way; however, for sufficiently small
perturbation, this would be postponed to arbitrarily long time.  For the time length considered
in our numerics, we can not distinguish between the linear and nonlinear growth.

We confirm this behavior numerically by
time evolution and Evans function analysis.
We derive also a general stability index similarly as in \cite{GZ,Z2,Go,Z4}
counting the parity of the number of unstable eigenvalues, hence
giving rigorous geometric necessary conditions for stability
in terms of the dynamics of the associated traveling-wave ODE;
see Section \ref{s:stabindex}.
Using the stability index, we show that homoclinics of the closely related
viscous Jin--Xin model are {\it always unstable}; see Section
\ref{s:JXequations}.
We point out also an interesting and apparently so far unremarked connection
between the stability index, a Melnikov integral for the homoclinic orbit
with respect to variation in the wave speed,
and geometry of bifurcating limit cycles (periodics), which
leads to a simple and easily evaluable rule of thumb connecting stability of
homoclinic orbits as traveling wave solutions of the associated
PDE to stability of an enclosed equilibrium as a solution of the
traveling-wave ODE; see Remark \ref{odestab}.

The examples found in this paper are notable as the first examples
of a solitary-wave solution of a second-order hyperbolic--parabolic
conservation or balance law for which the point spectrum is stable.
This raises the very interesting question whether an example could
exist with both stable point and essential spectrum, in which case
linearized and nonlinear orbital stability (in standard, not metastable sense)
would follow in standard fashion by the techniques of
\cite{MaZ1,MaZ4,Z2,JZ,JZN,LRTZ,TZ1}.
In Section \ref{s:num}
we present examples for various modifications
of the turbulent friction parameters (different from the ones
in \cite{N2,JZN}) that have stable point spectrum and ``almost-stable''
essential spectrum, or stable essential spectrum and ``almost-stable''
point spectrum, suggesting that one could perhaps find 
a spectrally stable example somewhere between.  However, we have up to now not
been able to find one for the class of models considered here.
Whether this is just by chance, or whether the conditions of stable
point spectrum and stable essential spectrum are somehow mutually
exclusive\footnote{For a similar dichotomy in a related, periodic context,
see \cite{OZ1}.}
is an extremely interesting open question, especially
given the fundamental interest of solitary waves in both theory
and applications.

Finally, we note that similar metastable phenomena have been observed
by Pego, Schneider, and Uecker \cite{PSU} for the related
fourth-order diffusive Kuramoto-Sivashinsky model
$$
u_t+\partial_x^4u+\partial_x^2u+\frac{\partial_x u^2}{2}=0,
$$
an alternative model for thin film flow down a ramp.
They describe asymptotic behavior of solutions of this model as dominated
by trains of solitary pulses, going on to state:
``Such dynamics
of surface waves are typical of observations in the inclined film problem
\cite{CD}, both
experimentally and in numerical simulations of the free-boundary Navier--Stokes
problem describing this system.''

That is, both our results and the results of \cite{PSU} seem to illustrate
a larger and somewhat surprising phenomenon deriving from
physical inclined thin-film flow
of asymptotic behavior dominated by trains of pulse solutions
which are themselves unstable.
We discuss this interesting issue, and the relation to stability of
periodic waves, in Section \ref{s:dyn},
suggesting a heuristic mechanism by which convectively unstable solitary
waves, when placed in a closely spaced array, can stabilize each other
by de-amplification of convected signals as they cross successive
solitary wave profiles.
We quantify this by the concept of ``dynamic spectrum'' of a solitary wave,
defined as the spectrum of an associated wave train obtained by periodic
extension, appropriately defined, of a solitary wave pulse.
In some sense, this notion captures the essential spectrum of the \textit{non-constant} portion of the profile.
This is in contrast to usual notion of essential spectrum which is governed
by the (often unstable) constant limiting states.
For the waves studied here, the dynamic spectrum is stable, suggesting strongly that long-wave periodic trains are stable as well.
This conjecture has since been verified in \cite{BJNRZ1,BJNRZ2}.

\medskip
{\bf Acknowledgement}: Thanks to Bj\"orn Sandstede for
pointing out the reference \cite{PSU}, and to Bernard Deconink for
his generous help in guiding us in the use of the SpectrUW package
developed by him and collaborators.
K.Z. thanks Bj\"orn Sandstede and Thierry Gallay for interesting conversations
regarding stabilization of unstable arrays, and thanks the Universities of Paris 7 and 13 for their warm
hospitality during a visit in which this work was partly carried out.
The numerical Evans function computations performed
in this paper were carried out
using the STABLAB package developed by Jeffrey Humpherys with
help of the first and last authors.

\section{The St. Venant equations with viscosity}\label{s:SVequations}

\subsection{Equations and setup}
The 1-d viscous St. Venant equations approximating shallow water flow on an
inclined ramp are
\ba \label{eqn:1Econslaw}
h_t + (hu)_x&= 0,\\
(hu)_t+ (h^2/2F+ hu^2)_x&= h- u|u|^{r-1}/h^s +\nu (hu_x)_x ,
\ea
where $1\le r\le 2$, $0\le s\le 2$, and
where $h$ represents height of the fluid, $u$ the
velocity average with respect to height,
$F$ is the Froude number, which here is the
square of the ratio between speed of the fluid and speed of gravity waves,
$\nu={\rm Re}^{-1}$ is a positive nondimensional viscosity
equal to the inverse of the Reynolds number,
the term $u|u|^{r-1}/h^s$ models turbulent friction along the bottom,
and the coordinate $x$ measures longitudinal distance along the ramp.
Furthermore, the choice of the viscosity term $\nu(hu_x)_x$ is motivated by the formal derivations
from the Navier Stokes equations with free surfaces.
Typical choices for $r$, $s$ are $r=1$ or $2$ and $s=0$, $1$, or $2$;
see \cite{BM,N1,N2} and references therein.
The choice considered in \cite{N1,N2,JZN} is $r=2$, $s=0$.

Following \cite{JZN}, we restrict to positive velocities $u>0$ and consider \eqref{eqn:1Econslaw} in
Lagrangian coordinates, in which case \eqref{eqn:1Econslaw} appears as
\ba \label{eqn:1conslaw}
\tau_t - u_x&= 0,\\
u_t+ ((2F)^{-1}\tau^{-2})_x&=
1- \tau^{s+1} u^r +\nu (\tau^{-2}u_x)_x ,
\ea
where $\tau:=h^{-1}$ and $x$ now denotes a Lagrangian marker
rather than physical location.  Note that working in Lagrangian coordinates was crucial
in developing the large-amplitude damping estimates necessary for the nonlinear stability
analysis of periodic roll waves in \cite{JZN}.

\br\label{pos}
\textup{
Since we study waves propagating down a ramp, there is no real
restriction in taking $u>0$.  However, we must remember to discard
any anomalous solutions which may arise for which negative values
of $u$ appear.
}
\er

We now consider the existence of solitary traveling wave solutions of \eqref{eqn:1conslaw}.
Denoting $U:=(\tau, u)$, consider a traveling-wave solution $U=\bar{U}(x -ct)$,
of \eqref{eqn:1conslaw}, which is seen to satisfy the traveling-wave ODE
\be \label{e:second_order}
\begin{aligned}
-c\bar\tau' - \bar u'&= 0,\\
-c \bar u'+ ((2F)^{-1}\bar\tau^{-2})'&=
1- \bar\tau^{s+1} \bar u^r +\nu (\bar\tau^{-2}\bar u')' .
\end{aligned}
\ee
Integrating the first equation of \eqref{e:second_order}
and solving for $\bar u= u(\bar\tau):= q-c\bar\tau$,
where $q$ is the resulting
constant of integration, we obtain a second-order scalar profile equation
in $\bar\tau$ alone:
\be \label{e:profile}
c^2 \bar\tau'+ ((2F)^{-1}\bar\tau^{-2})'=
1- \bar\tau^{s+1} (q-c\bar\tau)^r -c\nu (\bar\tau^{-2}\bar\tau ')' .
\ee

Note that nontrivial traveling-wave solutions of speed $c=0$
do not exist in Lagrangian coordinates,
as this would imply $\bar u\equiv q$, and
\eqref{e:profile} would reduce to a scalar first-order equation
\be \label{e:zerocprofile}
  \bar\tau'= F\bar\tau^3(\bar\tau^{s+1} q^r-1) ,
\ee
which has no nontrivial solutions with $\bar\tau>0$.  For non-zero values of $c$, however,
it is seen numerically that \eqref{e:profile} admits homoclinic orbits corresponding to
solitary wave solutions of \eqref{eqn:1conslaw}.  See also \cite{N1,N2}.  Furthermore, generically these are seem to emerge from
a hopf bifurcation from the equilibrium state, which generates a family of periodic orbits which terminate into
the bounding homoclinic.  

\subsection{The subcharacteristic condition and Hopf bifurcation}\label{s:hopf}

To begin, we consider the stability of the equilibrium solutions \eqref{eqn:1conslaw}.
First note that \eqref{eqn:1conslaw} is of $2\times 2$ viscous relaxation type
and can be written in the form
\be\label{relax}
U_t+H(U)_x-\nu (B(U)U_x)_x=\bp0\\g(U)\ep,
\ee
where $g(U)=1-\tau^{s+1} u^r$ and
\[
H(U)=\bp -u\\ \frac{1}{2F\tau^2}\ep,\quad B(u)=\left(
                                                 \begin{array}{cc}
                                                   0 & 0 \\
                                                   0 & \tau^{-2} \\
                                                 \end{array}
                                               \right).
\]
At equilibrium values $u=\tau^{-(s+1)/r}>0$,
the inviscid version of \eqref{relax}, obtained by setting $\nu=0$
in \eqref{relax}, has hyperbolic characteristics equal
to the eigenvalues $\pm \frac{\tau^{-3/2}}{\sqrt{F}}$ of $dH$, and
equilibrium characteristic $\big(\frac{s+1}{r}\big) \tau^{-(r+s+1)/r}$
equal to $df_*(\tau)$, where $f_*(\tau):=  g(\tau, u_*(\tau))$
and $u_*(\tau):= \tau^{-(s+1)/r}$ is defined by
$g(\tau, u_*(\tau))=0$. The subcharacteristic condition, i.e., the condition
that the equilibrium characteristic speed lie between the hyperbolic
characteristic speeds, is therefore
\be\label{subeq}
df_*(\tau)
=\left(\frac{s+1}{r}\right) \frac{u_*(\tau)}{\tau}
=\left(\frac{s+1}{r}\right) \tau^{-(r+s+1)/r}
\leq\frac{\tau^{-3/2}}{\sqrt{F}}=c_s(\tau).
\ee
For the common case $r=2$, $s=0$ considered in \cite{N2,JZN}, this reduces to $F\leq4$.  For $2\times 2$ relaxation systems such as the above, the subcharacteristic condition is exactly the condition that constant solutions be linearly stable, as may be readily verified by computing the dispersion relation using the Fourier transform.

\br\label{alternate}
Let $q$ and $c$ be fixed. Then the equilibria of the traveling-wave ODE \eqref{e:profile} are given by zeros of $g_*(\tau)=g(\tau, q-c\tau )$ (corresponding to $u_*(\tau)=q-c\tau$). Notice the relation
$dg_*(\tau)= g_\tau(\tau)-c g_u(\tau)= g_u(\tau)(df_*(\tau)-c)$.
Since $g_u<0$, by the relaxation structure of \eqref{relax}, the key quantity $df_*-c$
has a sign opposite to the one of $dg_*$.
Therefore, provided there is no multiple root of $g_*$, the sign of $df_*(\tau_0)-c$
alternates among the equilibria $\tau_0$ of the traveling-wave ODE \eqref{e:profile}.
\er

For the full system \eqref{relax} with viscosity $\nu> 0$, we now show some Fourier computations relating the subcharacteristic condition with both stability and Hopf bifurcation.
Let speed $c$ be fixed.  Then rewriting  \eqref{relax} 
in a moving frame $(x-ct,t)$ and linearizing about a constant solution
$(\tau, u)\equiv (\tau_0, u_0)$, with $\tau_0>0$ and $u_0=u_*(\tau_0)>0$, we obtain
the linear system $U_t + A_0 U_x=B_0 U + C_0U_{xx}$, where\footnote{Recall that $\tau_0^{s+1}u_0^r=1$.}
\be\label{consys}
A_0=\begin{pmatrix} -c & -1\\ -c_s^2 & -c\end{pmatrix}, \qquad
B_0=\begin{pmatrix} 0 & 0\\ -(s+1)/\tau_0 & -r/u_0\end{pmatrix}, \qquad
C_0=\begin{pmatrix} 0 & 0\\ 0 & \nu\tau_0^{-2}\end{pmatrix},
\ee
and $c_s=c_s(\tau_0)=\frac{\tau_0^{-3/2}}{\sqrt{F}}$ denotes the positive hyperbolic characteristic of the inviscid problem (see above). The corresponding dispersion relation between eigenvalue $\lambda$ and frequency $k$ is readily seen to be given by the polynomial equation
\ba \label{disrel}
0&=\det(\lambda I + A_0 ik-B_0+k^2C_0)\\
&=
\lambda^2 +
\left[\frac{r}{u_0}-2ick+\frac{\nu k^2}{\tau_0^2}\right]\lambda+
ik\left[\frac{s+1}{\tau_0}-\frac{cr}{u_0}+ik (c^2-c_s^2)-\frac{c\nu k^2}{\tau_0^2}\right].
\ea
Inspecting stability, we Taylor expand the dispersion relation \eqref{disrel} about an eigenvalue $\lambda$ with $\Re (\lambda)=0$.
Note that if $\lambda$ is such an eigenvalue corresponding to a non-zero frequency $k$, we may assume $\lambda=0$
by replacing $c$ with $c-\Im(\lambda)/k$. Let us then focus on $k=0$ or $\lambda=0$ and $k\neq0$.

Eignevalues corresponding to $k=0$ are $0$ and $-r/u_0$. Solving the dispersion equation \eqref{disrel} about $(\lambda,k)=(0,0)$ with $\lambda(k)$ and differentiating, one finds
$$
\lambda'(0)=-i\left[\left(\frac{s+1}{r}\right)\frac{u_0}{\tau_0}-c\right]
$$
and
$$
\frac12\lambda''(0)
=\frac{u_0}{r}\left[\left(i\lambda'(0)+c\right)^2-c_s^2\right]\\
=\frac{u_0}{r}\left[\left(\left(\frac{s+1}{r}\right)\frac{u_0}{\tau_0}\right)^2-c_s^2\right].
$$
so that failure of the subcharacteristic condition implies spectral instability. Furthermore, it follows that any solitary wave solution of \eqref{eqn:1conslaw} for which the limiting value $\lim_{x\to\pm\infty}\tau(x)$ violates the subcharacteristic condition \eqref{subeq} will have unstable essential spectrum.

We now look for eigenvalues $\lambda=0$ for $k\neq0$ by solving
\be \label{e:linprof}
0=\frac{s+1}{\tau_0}-\frac{cr}{u_0}+ik (c^2-c_s^2)-\frac{c\nu k^2}{\tau_0^2}.
\ee
Such a $k$ exists if and only if $c=c_s$ and $\left(\frac{s+1}{r}\right)\frac{u_0}{\tau_0}>c_s$,
and then the solutions are $\pm k_H$ with
$$
k_H=\sqrt{\frac{\tau_0^2r}{c_s\nu u_0}\left[\left(\frac{s+1}{r}\right)\frac{u_0}{\tau_0}-c_s\right]}.
$$
Now, solving the dispersion equation \eqref{disrel} about $(\lambda,k)=(0,\pm k_H)$ with $\lambda(k)$ and differentiating, one finds
\be \label{Hopftransverse}
\lambda'(\pm k_H)=\frac{2c_s\nu}{\tau_0^2}\frac{i k_H^2}{r/u_0+\nu k_H^2/\tau_0^2\mp2ic_sk_H}\notin i\R
\ee
yielding again instability. See Figure \ref{ConstantSpec} for a typical depiction of the triple degeneracy of the essential spectrum at $\lambda=0$ for an unstable constant state when $c=c_s$.

\begin{figure}[htbp]
\begin{center}
\includegraphics[scale=.45]{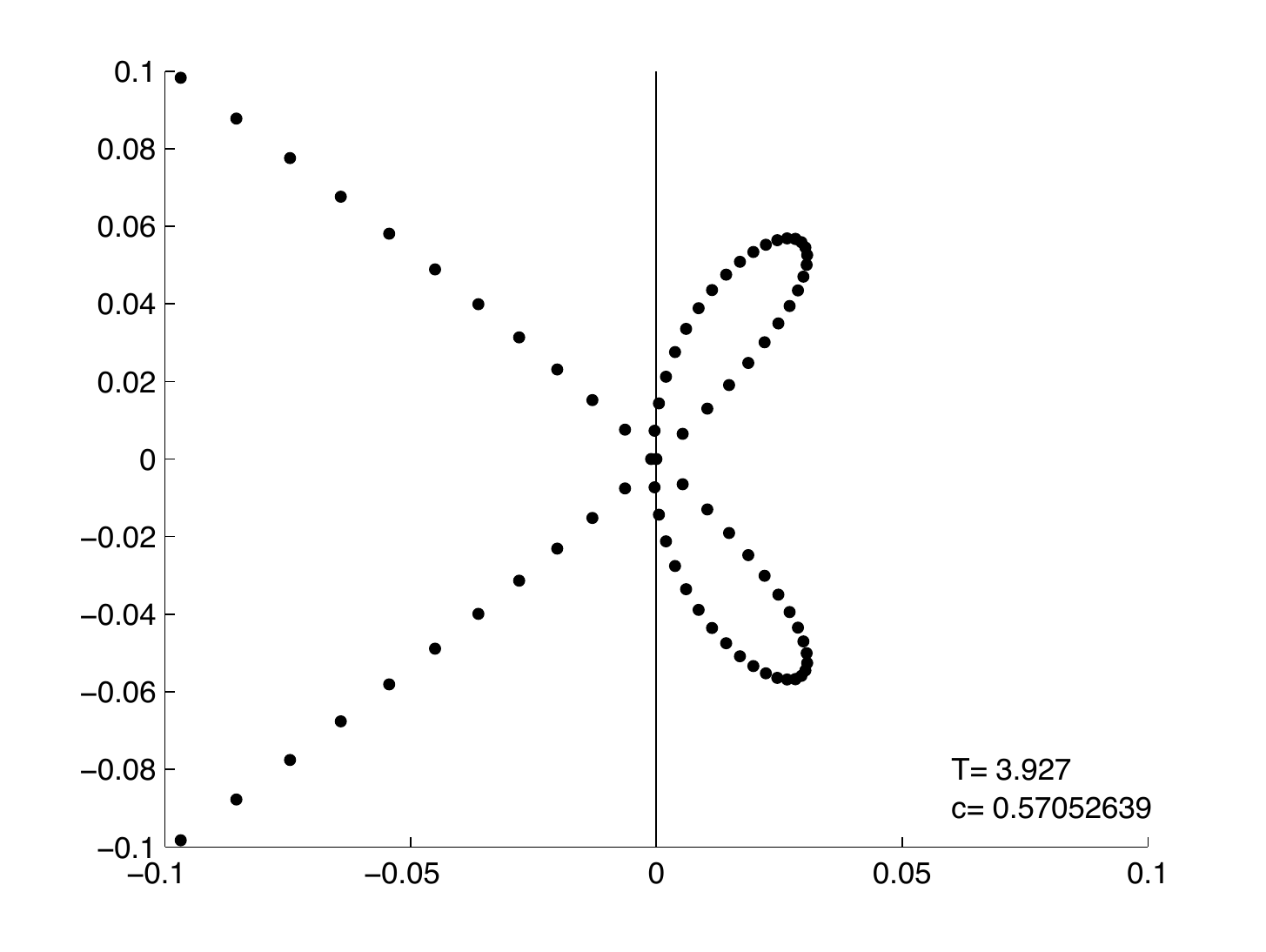}
  \caption{The essential spectrum for the unstable equilibrium solution $(\tau_0,u_0)$ of \eqref{e:profile} at a Hopf bifurcation $c=c_s$ is shown, where
  here we take $u_0=0.96$, $q=u_0+c/u_0^2$, $\nu=0.1$, $r=2$, $s=4$, $F=6$, and $c=c_s=0.57052639$.  Notice the triple degeneracy
  of the $\lambda=0$ eigenvalue in this case.}
   \label{ConstantSpec}
\end{center}
\end{figure}

Actually equation \eqref{e:linprof} is the linearization of the profile equation \eqref{e:profile} (with $q=u_0+c\tau_0$) about $\tau_0$, expressed in Fourier variables, so that previous pragraph translates to the {\it Hopf bifurcation conditions} for the generalized St. Venant equations \eqref{eqn:1conslaw}
\be\label{e:hopfcond}
c=c_s(\tau_0)= \frac{\tau_0^{-3/2}}{\sqrt{F}}
\quad \hbox{\rm and }\;
df_*(\tau_0)>c_s(\tau_0).
\ee
Moreover $k_H$ is the limiting frequency associated to the Hopf bifurcation and relation \eqref{Hopftransverse} will yield instability for close-by periodic travelling-waves generated through the Hopf bifurcation.  See \cite{JZN} for an alternate derivation of these bifurcation conditions.

\br\label{unstableregime}
\textup{
Experiments of \cite{N2,BJNRZ1,BJNRZ2} in case $r=2$, $s=0$ indicate that, when $F>4$, there is a smooth family of periodic waves parametrized by period, which increase in amplitude as period is
increased, finally approaching a limiting homoclinic orbit as period goes to infinity; see \cite{N2,BJNRZ1,BJNRZ2}. Likewise, in the study of a related artificial viscosity model with $r=2$, $s=0$
\cite{HC}, periodic and homoclinic solutions {\it were found in the regime of unstable constant solutions, $F>4$}. More precisely, in \cite{HC} this very configuration and generation of the periodics and homoclinics were found analytically near where the nonlinear center and saddle point coalesce. The authors were then able to analytically follow these configurations to obtain a global description of phase space.
}
\er

\br\label{repellor}
\textup{
Let us further examine the linearized profile equation. For this equation, if $c(df_*(\tau_0)-c)<0$
then the equilibrium $\tau_0$ is a saddle point ; whereas if $c(df_*(\tau_0)-c)>0$
then if $c(c^2-c_s(\tau_0)^2)<0$ the equilibrium $\tau_0$ is a repellor, either of spiral or real type, and if $c(c^2-c_s(\tau_0)^2)>0$ the equilibrium $\tau_0$ is an attractor, either of spiral or real type.  In particular, by the alternating property, the repellor case describes behavior at a single equilibrium $\tau_1$ trapped within a homoclinic (asymptotic to $\tau_0$) if and only if there stands $c(c^2-c_s(\tau_1)^2)<0$. For the example wave considered in Section \ref{s:num} we have $r=2$, $s=0$, $c=.78$, $F=6$, $q=1.78$, $\tau_0=1$ and $c_s(\tau_0)=1/\sqrt{2F}<.289<.78$ at the saddle equilibrium, but, at the enclosed equilibrium $\tau_1<1/2$, $c_s(\tau_1)=\tau_1^{-3/2}/\sqrt{2F}> .816>.78$.
}
\er

\section{Analytical stability framework}\label{s:framework}

\subsection{Linearized eigenvalue equations and convective instability}\label{s:friction}

We now begin our study of the stability of a given homoclinic solution $U(x,t)=\bar U(x-ct)$
satisfying $\lim_{x\to \pm \infty}=U_0$ to localized perturbations in $L^2(\RM;\RM^2)$.  To this
end, notice that
changing to co-moving coordinates, linearizing about $\bar U=(\bar \tau,\bar u)$, and taking
the Laplace transform in time we obtain the linearized eigenvalue equations
\ba \label{e:homlinII}
\lambda \tau-c\tau' - u'&= 0,\\
\lambda u-cu'
-(\bar \tau^{_-3}(F^{-1}- 2\nu \bar u_x)\tau)' &=
-(s+1)\bar \tau^s\bar u^r \tau
- r\bar \tau^{s+1}\bar u^{r-1} u
+\nu (\bar \tau^{-2}u')' .
\ea
As we are considering localized perturbations, we consider the $L^2$ spectrum of the eigenvalue
equations \eqref{e:homlinII}.  Further, we say the underlying solitary wave $\bar U$ is spectrally stable
provided that the $L^2$ spectrum of \eqref{e:homlinII} lies in the stable half
plane $\{\Re\lambda\leq 0\}$ and is spectrally unstable otherwise.

In order to analyze the spectral stability of the underlying solution $\bar U$, we consider the point
and essential spectrum of the linearized equations \eqref{e:homlinII} separately.
Here, we define the point spectrum in the generalized sense of \cite{PW2}
(including also resonant poles) as zeros of an associated Evans function (description just below),
and stability of the point spectrum as nonexistence of such zeros with nonnegative real part, other than a single
translational eigenvalue always present at frequency $\lambda=0$.
Equivalently \cite{Sat}, this notion of generalized point spectrum
can be described in terms of usual eigenvalues
of the linearized operator with respect to an appropriate weighted
norm; see \cite{TZ2} for further discussion.
Furthermore, we define the essential spectrum in the usual complementary sense: that is,
as the relative complement of the point spectrum (in the generalized sense described above) in the $L^2$ spectrum
of the linearized equations \eqref{e:homlinII}.

Concerning the structure of the essential spectrum, 
we recall a  classical theorem of Henry \cite{He,GZ} which states that the essential
spectrum of the linearized operator \eqref{e:homlinII} about the wave both includes and is bounded to the right
in the complex plane
by the right envelope of the union of the essential spectra of the linearized operators about the constant
solutions 
corresponding to the left and right end states $\lim_{x\to\pm\infty}U(x)$, which in our case
agree with the common value $U_0$.  This in turn may
be determined by a Fourier transform computation as in Section \ref{s:hopf}.
Consulting our earlier computations, we see that in the common case
$r=2$, $s=0$, constant solutions are {\it unstable} in the region of
existence of solitary waves, and hence solitary waves in this case always have
{\it unstable essential spectrum} (see Remark \ref{unstableregime}), and
hence are linearly unstable with respect to standard (e.g., $L^p$ or $H^s$) norms.

As discussed in \cite{PW2,Sat}, however, such instabilities are often
of ``convective'' nature, meaning that growing perturbations are simultaneously
swept away from, or ``radiated'' from the solitary wave profile, which
itself remains, at least to linear order, intact.
This corresponds to the situation that generalized
 {point spectrum}, governing
near-field behavior, is stable, while essential spectrum, governing
far-field behavior is unstable.
The phenomenon of convective instability can
%
%
 be captured
at a linearized level by the introduction of an appropriate weighted norm
\cite{Sat,PW1}, with respect to which the essential spectrum is shifted
into the negative half-plane and waves are seen (by the Hille--Yosida
Theorem) to be linearly time-exponentially stable.
Alternatively, see the pointwise description 
obtained by stationary phase estimates in \cite{OZ1}
of the far-field behavior as of an oscillatory
time-exponentially growing Gaussian wavepacket convected with respect
to the wave.  (See also \cite{AMPZ2}.)

Next, we continue our discussion by developing the necessary tools to analyze the point spectrum
of \eqref{e:homlinII}.  Notice that, as stated in the introduction, spectral stability of {\it both} point and essential spectrum,
should it occur, can be shown to imply linearized and nonlinear stability
in the standard time-asymptotic sense
by the techniques of  \cite{MaZ1,MaZ3,MaZ4,LRTZ,TZ1,JZN}.
However, as also noted in the introduction, we have so far not found a case in
which these conditions coexist.

\subsection{Construction of the Evans function}\label{s:evans_framework}

In order to describe the point spectrum (defined above in the generalized sense), we
outline here the construction of the Evans function corresponding to the
linearized eigenvalue equations \eqref{e:homlinII}.  First, notice that setting $W:=(\tau, u, \bar \tau^{-2} u')^T$, and recalling that
$\bar u=(q-c\bar \tau)$, we may write \eqref{e:homlinII} as a first-order system of the form $W'=A(x,\lambda)W$,
which is suitable for an Evans function analysis, where
\be\label{homAII}
A:=\bp
\lambda/c& 0 & -\bar \tau^2/c\\
0 & 0 &  \bar \tau^2\\
\frac{(s+1)\bar \tau^s (q-c\bar \tau)^r -\bar \alpha_x-\bar \alpha \lambda/c}{\nu} &
\frac{\lambda + r\bar \tau^{s+1}(q-c\bar \tau)^{r-1}}{\nu} & \frac{-c\bar \tau^2+
\bar \alpha \bar \tau^2/c}{\nu}
\ep,
\ee
\be\label{alphaII}
\bar \alpha:= \bar \tau^{-3}(F^{-1}+ 2c\nu \bar \tau_x),
\qquad
\bar \alpha_x = -3 \bar \tau^{-4}\bar \tau_x (F^{-1}+2 c\nu \bar \tau_x)
+2\bar \tau^{-3}c \nu \bar \tau_{xx}.
\ee
As with the essential spectrum, the point spectrum can essentially be described by analyzing the behavior
of the system \eqref{homAII} evaluated at the constant end state $U_0$ of the underling solitary wave.
To begin, we make the following definition which extends the notion of hyperbolicity of the
limiting first order system \eqref{homAII}.

\begin{definition}
Following \cite{AGJ,GZ},
for a generalized eigenvalue equation $W'=A(x,\lambda)W$, with
$$
\lim_{x\to \pm \infty}A(x, \lambda)=A_0(\lambda),
$$
we define the {\rm region of consistent splitting} to be the connected component
of real plus infinity of the set of $\lambda$ such that
$A_0$ has no center subspace, that is, the
stable and unstable subspaces $\mathcal{S}(A_0)$ and $\mathcal{U}(A_0)$
of $A_\pm$ are of constant dimensions summing to the dimension of the
entire space and maintain a spectral gap with respect to zero.
\end{definition}

We say that an eigenvalue problem satisfies consistent splitting if
the region of consistent splitting includes the entire punctured unstable half-plane
$\{\lambda\,|\,\Re \lambda\geq 0\}\setminus\{0\}$ of interest in the study of stability.\footnote{Recall, \cite{GZ,MaZ3,JZN},
that the special point $\lambda=0$ may be adjoined by analytic extension.}  The notion of
consistent splitting contains all the necessary information in order to define an Evans function.
Furthermore, the notion of consistent splitting has an immediate consequence concerning the
stability of the limiting end states.  

\begin{lemma}[\cite{AGJ,GZ,MaZ3}]
Consistent splitting is equivalent to spectral stability of
the constant solution $U\equiv U_0$, defined as nonexistence
of spectra of the constant-coefficient linearized operator
about $U_0$ on the punctured
unstable half-plane $\{\,\lambda\,|\,\Re \lambda\geq 0\}\setminus\{0\}$.
\end{lemma}

\begin{proof}
This follows from the standard observation
\cite{GZ,ZH,MaZ3,Z1} that the set of $\lambda$ for
which $A_0(\lambda)$ possesses a pure imaginary eigenvalue
$\mu=ik$ is given by the linear dispersion relations
$\lambda=\lambda_j(k)$ for the associated constant-coefficient
equation $U_t=LU:= BU_{xx} -AU_x +CU$, defined by the roots of
$\det(\lambda I-k^2B-ikA+C)=0$.
\end{proof}

In the common case $r=2$, $s=0$, the Fourier transform computations of Section \ref{s:hopf} show that
constant solutions are unstable in the regime of existence of homoclinic orbits\footnote{At least, for those
generated through a Hopf bifurcation.  Recall, however, that all homoclinic orbits studied here are obtained
in this way.} and thus consistent splitting must fail in this case.  Thus, while the notion of consistent splitting
is sufficient for the construction of an Evans function, it is insufficient for our purposes.
For our applications then, we need a generalization of the notion of consistent splitting which
still incorporates the essential features necessary for the construction of an Evans function.  Such
a generalization was observed in \cite{PW2}, which we now recall.

%


\begin{definition}
We define the
{\rm extended region of consistent splitting} to be the largest superset
of the region of consistent splitting on which the total eigenspaces associated
with the stable (resp. unstable) subspaces of $A_0$ maintain a spectral
gap {\rm with respect to each other}, that is, for which the eigen-extensions
$\hat {\mathcal{S}}$ and $\hat { \mathcal{U}}$
of the stable and unstable subspaces
$\mathcal{S}(A_0)$ and $\mathcal{U}(A_0)$ defined for $\lambda$
near real plus infinity satisfy the condition that
the maximum real part of eigenvalues of $\hat { \mathcal{S}}$ is strictly
smaller than the minimum real part of of eigenvalues of $\hat{ \mathcal{U}}$.
\end{definition}

We say that an eigenvalue problem satisfies extended consistent splitting if the region of extended consistent splitting includes the entire punctured unstable half-plane\linebreak $\{\lambda|\Re \lambda\geq 0\}\setminus\{0\}$. Notice that the Fourier transform computations of Section \ref{s:hopf} show that extended consistent splitting holds
in a neighborhood of the origin\footnote{Away from the origin, extended consistent splitting can be checked numerically in the course
of performing the Evans function computations outlined in Section \ref{s:num}.  In particular, the numerical code
for the Evans computations is written in such a way that failure of extended consistent splitting signals an error.}
in the common case $r=2$, $s=0$ while, as mentioned above, consistent splitting does not.
Next, under the assumption of extended consistent splitting we show that it is possible to relate the large $x$ behavior of the
first order system $W'=A(x,\lambda)W$, with $A$ given as in \eqref{homAII}, to the behavior of limiting system
$W'=A_0(\lambda)W$, where the coefficient matrix $A_0$ is defined as above.

%


\begin{lemma}[\cite{GZ,ZH,Z1}]\label{limitingsystem}
Let $A(x,\lambda)$ be analytic in $\lambda$ as a function into $C^0(x)$,
exponentially convergent to $A_0$ as $x\to \pm \infty$,
and satisfying extended consistent splitting.
Then, there exist globally defined
bases $\{W_1^-, \dots, W_k^-\}$ and
$\{W_{k+1}^+, \dots, W_n^+\}$,
analytic in $\lambda$ on $\{\Re \lambda\geq 0\}\setminus\{0\}$,
spanning the manifold of solutions of $W'=A(x,\lambda )W$
approaching exponentially in angle as $x\to -\infty$ (resp.
$x\to +\infty$)
to $\hat {\mathcal{U}}(A_0)$ and $\hat {\mathcal{S}}(A_0)$,
where $k=\dim \hat {\mathcal{U}}(A_0)$ and $n=\dim U$; more precisely,
\be\label{asymptotics}
W_j^\pm \sim e^{\pm A_0 x}V_j^\pm
\quad \hbox{\rm as }\; x\to \pm \infty
\ee
for analytically chosen bases $V_j^\pm$ of
$\hat {\mathcal{U}}(A_0)$ and $\hat {\mathcal{S}}(A_0)$.
For the eigenvalue equations considered here, these bases
(both $W_j^\pm$ and $V_j^\pm$) extend
analytically to $\Re \lambda \le -\eta<0$.
\end{lemma}

\begin{proof}
The first, general result follows by the conjugation lemma of
\cite{MeZ} (see description, \cite{Z1}) asserting that exponentially
convergent variable-coefficient ODE on a half-line
 may be converted by an exponentially
trivial coordinate transformation to constant-coefficient ODE
$Z'=A_0 Z$, together with a lemma of Kato \cite{K,Z6} asserting
existence of globally analytic bases $\{V_j^\pm\}$
for the range of an analytic projection,
the eigenprojections being analytic due to spectral gap.
Extension to $\lambda=0$ (and thus $\Re \lambda \le -\eta<0$ for
some $\eta$) then follows by a matrix perturbation expansion at
$\lambda=0$ verifying analytic extension of the projectors.
See \cite{GZ,MaZ3,Z1} for the now-standard details of this argument.
\end{proof}

By Lemma \ref{limitingsystem} then, solutions of the full linearized system $W'=A(x,\lambda)W$
behave asymptotically like the solutions of the limiting system $W'=A_0(\lambda)W$.
Thus, if $\lambda$ is in the point spectrum (in the usual $L^2$ sense) of \eqref{e:homlinII} the
corresponding vector solution of full linearized system should decay as $x\to\pm\infty$
corresponding to a connection between the stable and unstable subspaces of $A_0$.
For the linearized St. Venant eigenvalue equations considered here,
computations like those of Section \ref{s:hopf} show that on the region of consistent splitting the limiting coefficient matrix
$A_0(\lambda)$ 
has a one-dimensional stable subspace\footnote{This dimension can be found near the origin or near real plus infinity
and is, by definition, constant in the region of consistent splitting.}  and hence
Lemma \ref{limitingsystem} applies with $k=2$ and $n=3$; as with extended consistent splitting, this is something that can be
numerically verified away from $\lambda=0,\infty$.  It follows that $\lambda$ belongs to the
point spectrum of \eqref{e:homlinII} provided the spaces $\textrm{span}\{W_1^-,W_2^-\}$ and $\textrm{span}\{W_3^+\}$
intersect non-trivially, which leads us to our definition of the Evans function.

\begin{definition}\label{evansdef}
For the linearized St. Venant eigenvalue equations, we
define the Evans function as
\be\label{eq:evansdef}
D(\lambda):=
\frac{\det (W_1^-, W_2^-,W_3^+)|_{x=0}}
{\det(V_1^-,V_2^-,V_3^+)},
\ee
where $V_j^\pm$ are the limiting directions of $W^\pm_j$.  Furthermore, we say $\lambda$
belongs to the point spectrum of \eqref{e:homlinII} provided $D(\lambda)=0$.
\end{definition}

\br
\textup{
Note that definition \eqref{eq:evansdef} is independent of
the choice of basis elements $V_j^\pm$, $W_j^\pm$, so
long as \eqref{asymptotics} holds.
Thus, it is not necessary in practice to choose the $V_j^\pm$ analytically,
or even continuously.
This
way of normalizing
the Evans function for homoclinics by considering
the limiting vectors $V_{j}^\pm$ seems to not only be new to the literature,
but also quite convenient from a computational standpoint; see calculations below.
}
\er

When consistent splitting holds, the basis elements $W_j^\pm$
span the manifolds of solutions decaying as $x\to +\infty$
(resp. $x\to -\infty$), and so vanishing of $D$ is associated
with existence of an exponentially decaying eigenfunction and
roots $\lambda$ correspond to standard eigenvalues.
When extended consistent splitting holds but consistent splitting
does not, then some of the basis elements may be nondecaying or even
growing at infinity, and so roots correspond rather to ``resonant poles''.
See \cite{PW2} for further discussion.  Thus, as mentioned before, our definition of point
spectrum is in a generalized sense including all roots (including resonant poles) of the Evans function as defined in Definition \ref{evansdef}.
However, notice that if one has good essential spectrum and bad point (Evans) spectrum, then the unstable Evans spectra
lies in the region of consistent splitting (excepting the degenerate case $\lambda=0$) and hence corresponds to a spectral instability to localized perturbations.  In particular,
in this case is is possible to conclude nonlinear instability of the solitary wave by arguments like that of \cite{Z8,Z9}.
On the other hand, if one has both bad essential and bad point (Evans) spectrum, then the instability depends
on which component of the spectra has larger real part; if the unstable point spectrum is farther out, then
again this corresponds to an unstable localized eigenvalue, while if the essential spectrum is farther out, then we are
essentially in the convective instability case briefly discussed in the introduction.  The latter convective instability
will be discussed further in Section \ref{s:dyn}.

\subsection{The stability index}\label{s:stabindex}
Next, we use the Evans function to seek conditions which imply the existence of unstable point (Evans) spectrum.
First, however, we need the following lemma which guarantees that our construction of the Evans function
in the previous section is valid along the positive real axis.

\bl\label{l:split}
Let $\tau_0$ be an equilibrium solution of the profile ODE \eqref{e:profile}.
Then, the open positive real axis is contained in the set of
consistent splitting for the linearized St. Venant equations
\eqref{e:homlinII} if and only if $cdf_*(\tau_0) \le c^2$ or $c^2-c_s^2\le 0$.
In any case, there exists a neighborhood of the origin
which lies in the region of extended consistent splitting.
\el

\br
Notice by Remark \ref{repellor}, the condition $c(df_*(\tau_0)-c)<0$ is equivalent to the equilibrium $\tau_0$ being a saddle point of the
profile ODE \eqref{e:profile}.  It follows that, given any homoclinic solution of \eqref{e:profile},
the set $(0,\infty)$ is contained in the set of consistent splitting and hence any positive root of the
Evans function corresponds to a spectral instability of the underlying wave to localized perturbations.
\er

\begin{proof}
Recall that the rightmost boundary of the essential spectrum is precisely the boundary of the
region of consistent splitting.  Thus, to establish the first claim it is sufficient to prove that the essential spectrum does
not intersect the positive real axis if and only if
$cdf_*(\tau_0) \le c^2$ or $c^2-c_s^2\le 0$.

This in turn follows by the Fourier analysis of Section \ref{s:hopf}.
From \eqref{disrel} we see that that the essential spectrum $\lambda=\lambda(k)$ intersects the
real axis if and only if
\[
-2ck\lambda+k\left(\frac{s+1}{\tau_0}-\frac{cr}{u_0}-\frac{c\nu k^2}{\tau_0^2}\right)=0.
\]
When $k\neq 0$ such roots must satisfy
\begin{equation}\label{e:realLambda}
\lambda=\frac{1}{2c}\left(\frac{s+1}{\tau_0}-\frac{cr}{u_0}-\frac{c\nu k^2}{\tau_0^2}\right)
\end{equation}
and hence a sufficient condition that $\lambda \le 0$ is that
\[
\frac{s+1}{c\tau_0} \le \frac{r}{u_0}.
\]
Rearranging using the fact that $u_0=\tau_0^{-(s+1)/r}$, we find that this last condition is equivalent to $\frac{1}{c}df_*(\tau_0)\le 1$, which
is equivalent to $c(df_*(\tau_0)-c)\le 0$.
Therefore, under this sign condition, we conclude that the only real roots $\lambda$ of \eqref{disrel} that can occur when $k\neq 0$
must be negative.  But, when $k=0$ we have the explicit roots $\lambda=0$ and $\lambda=\frac{-r}{u_0}$, which are also non-negative and hence
yields the desired contradiction.

Likewise, for $k\neq 0$, substituting \eqref{e:realLambda}
into the real part of the dispersion relation \ref{disrel} yields the relation
\[
\frac{\lambda}{2}\left(\frac{s+1}{c\tau_0}+\frac{r}{u_0}+\frac{\nu k^2}{\tau_0^2}\right)=k^2(c^2-c_s^2),
\]
which is a contradiction if $\lambda>0$ and $c(df_*(\tau_0)-c)>0$ and $c^2-c_s^2\le 0$.  Thus, we also see that the open positive real axis
lies in the region of consistent splitting whenever $c^2-c_s^2\le 0$.

If, on the other hand, $c(df_*(\tau_0)-c)> 0$ and $c^2-c_s^2> 0$,
then, necessarily $|df_*|>|c_s|$, hence the subcharacteristic condition
is violated and there exist essential spectra with strictly real part.
Fixing $\tau_0$, hence $df_*$ and $c_s$, now vary the wave speed $c$.
For $c=c_s$, we have as already noted the Hopf configuration
shown in Figure \ref{ConstantSpec} of a triple degeneracy with three
distinct roots $k=0,\pm k_*$ at $\lambda=0$, and (by the sufficient condition
already established), the positive real axis is contained in the region
of consistent splitting.
Recalling that changes in the wave speed $c$ by an amount $\Delta c$
corresponds to a translation in the roots of the dispersion
relation along the imaginary axis by an amount proportional to $i(\Delta c)k$,
we find that for $c^2\gtrless c_s^2$, the triple singularity breaks up
into three distinct crossings of the imaginary axis, one at $\lambda=0$
and two at $\pm ik_*(c-c_s)$. For $c^2<c_s^2$, we have already shown
that the real axis remains in the region of consistent splitting.
See Figure \ref{venant_evan_homoclinic}(e),which shows an example
of the essential spectrum in a case when $c>df_*(\tau_0)>c_s>0$.  
When $c(df_*(\tau_0)-c)>0$ and 
$c^2>c_s^2$,  therefore, the roots move in the opposite direction
along the real axis, and the essential spectrum must intersect
the positive real axis as in Figure \ref{ConstantSpec-bad}.
(Alternatively, we may see this by expansion of the dispersion relation
in the vicinity of $\lambda=0$ and $c=c_s$.)
This verifies the first claim.

Finally, notice that expanding the dispersion relation \eqref{disrel} with the substitution $\mu=ik$, it follows that the
matrix $A(x;0)\equiv A_0(0)$ from the Evans system about the equilibrium solution $(\tau_0,u_0)$ with $\lambda=0$
possess a one-dimensional center subspace.  As the spectral gap between the stable and unstable subspaces of $A_0(0)$
persist under small perturbation, it follows that the corresponding matrices $A_0(\lambda)$ satisfy extended consistent splitting
for $|\lambda|$ sufficiently small, which verifies the final claim.
\end{proof}


\begin{figure}[htbp]
\begin{center}
\includegraphics[scale=.45]{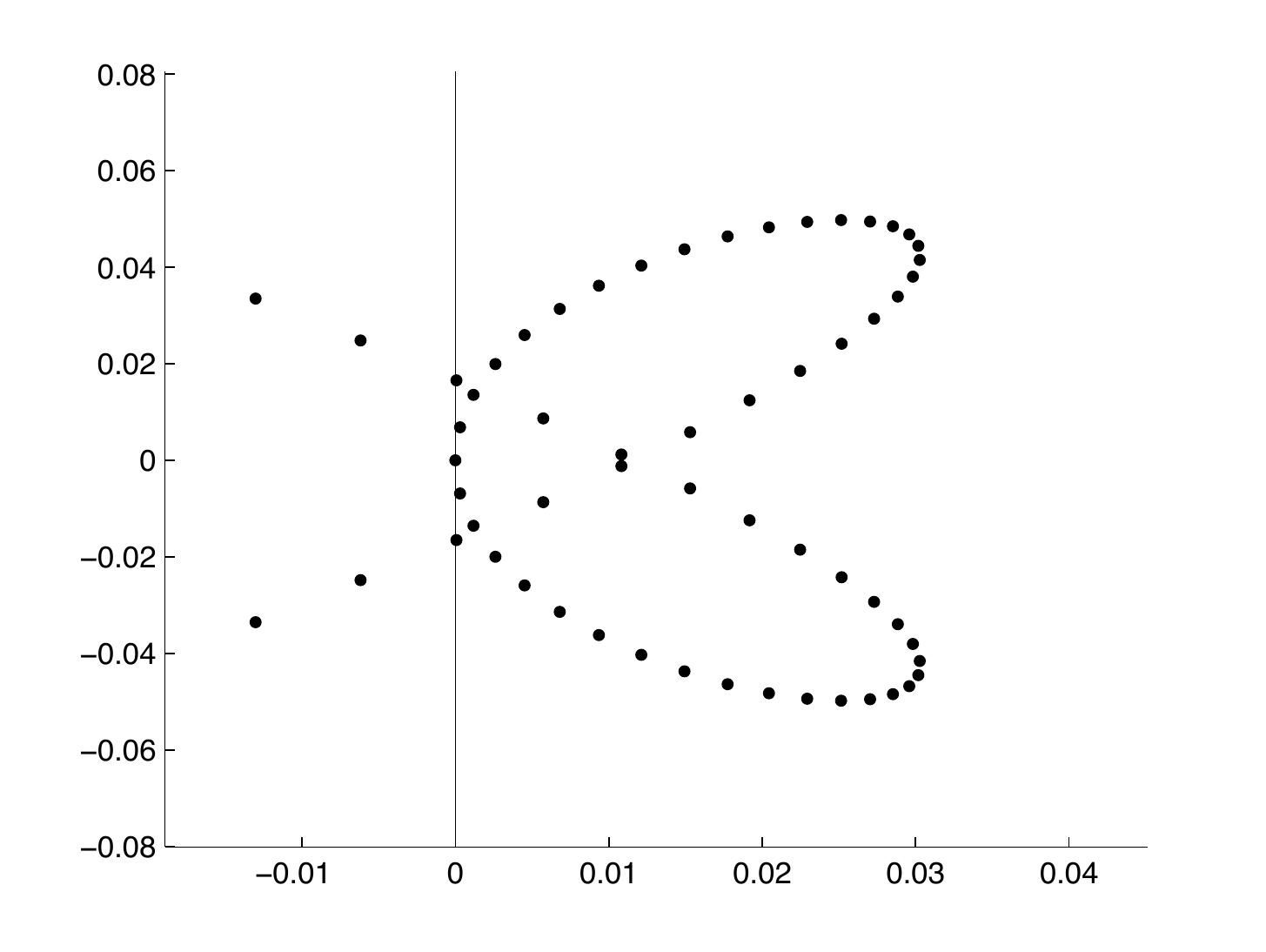}
  \caption{This figure depicts the essential spectrum for the constant state $\tau_0\approx 1.08507$ for the common case $(r,s)=(2,0)$
  when $c=0.53$, $q\approx 1.53509$, $\nu=0.1$, and $F=6$.  In particular, it is easily seen that $c(df_*(\tau_0)-c)>0$ and $c^2>c_s^2\approx(0.361)^2$ in this case.
  Notice that the essential spectrum intersects the positive real axis}
   \label{ConstantSpec-bad}
\end{center}
\end{figure}

Following \cite{GZ} then, we can analyze the Evans function on a set of the form $(-\eps,\infty)$ for some $\eps>0$.
In particular, by comparing high and low (real) frequency behavior of the Evans function, may compute a mod-two stability index giving
partial stability information in terms of geometric properties of the
flow of the traveling-wave ODE: specifically,
the sign of a Melnikov derivative with respect to wave speed.  To this end, we begin by analyzing the
high (real) frequency asymptotics of the Evans function.

\bl
$\sgn D(\lambda)\to 1$ as $\lambda \to +\infty$ for $\lambda$ real.
\el

\begin{proof}
This may be checked by homotopy
to an easy case, followed by direct computation.
Or, we may observe that asymptotic behavior is determined by
the principal part
$$
\bp \frac{\lambda}{c}& 0 & 0\\
0 & 0 &  \bar \tau^2\\
* & \frac{\lambda}{\nu} & 0\\
\ep
$$
of $A(x,\lambda)$, hence by the Tracking Lemma
of \cite{GZ,ZH,MaZ3},
$\det (W_1^-,W_2^-,W_3^+)_{x=0}$ is asymptotic to a positive multiple
of $\det (V_1^-,V_2^-,V_3^+)$.
See the computations of Section \ref{s:track} for details.
\end{proof}

In order to derive an instability index, we compare the above high frequency behavior with the low-frequency
behavior near $\lambda=0$.  First, notice that differentiating \eqref{e:second_order} with respect to $x$ implies
that $U_x$ satisfies the linearized eigenvalue equations \eqref{e:homlinII} with $\lambda=0$ and hence $D(0)=0$.
It follows negativity of $D'(0)$ is sufficient for the existence of a real, positive, unstable
element of the point (Evans) spectrum of \eqref{e:homlinII}.  Notice, however, that $D'(0)<0$ always implies either the existence
of an unstable localized eigenvalue (if this point lies in the region of consistent splitting), or else
the presence of essential spectrum on the positive real axis.  Therefore, given that $D'(0)<0$ a conclusion
can always be made concerning the stability of the underlying solitary wave: a time-exponential instability
to localized perturbations or else a convective-time-oscillatory instability (as described in the introduction and in Section \ref{s:dyn} below).
It turns out that the sign of this
derivative is intimately related to the underlying geometry of the ODE flow induced by the profile ODE.
Indeed, for values of $(c,q)$ such that there is a smooth saddle equilibrium
at $(\tau_0(c,q),0)$ in the traveling wave ODE, let
$\bar U^+(c,q;x)$ parameterize the stable manifold at
$(\tau_0,0)$ and $\bar U^-(c,q;x)$ the unstable manifold.
Define the {\it Melnikov separation function }
\be\label{mel}
d( c,q)=
\det \bp \bar \tau_x & \bar \tau^+- \bar \tau^-\\
(\bar \tau_x)' & (\bar \tau^+- \bar \tau^-)'\\
\ep|_{x=0},
\ee
and notice that
this represents a signed distance of $\bar U^-$ from $ \bar U^+$ along a normal section
at $\bar U|_{x=0}$ oriented in the direction of $\bar U_x^\perp|_{x=0}$ with respect to the right hand rule; see Figure~\ref{melnikov}.  As
seen in the next lemma, the sign of the quantity $D'(0)$ is determined precisely by whether $d(c,q)$ is an
increasing or decreasing function of the wave speed $c$.


\begin{figure}[htbp]
\begin{center}
\includegraphics[scale=.45]{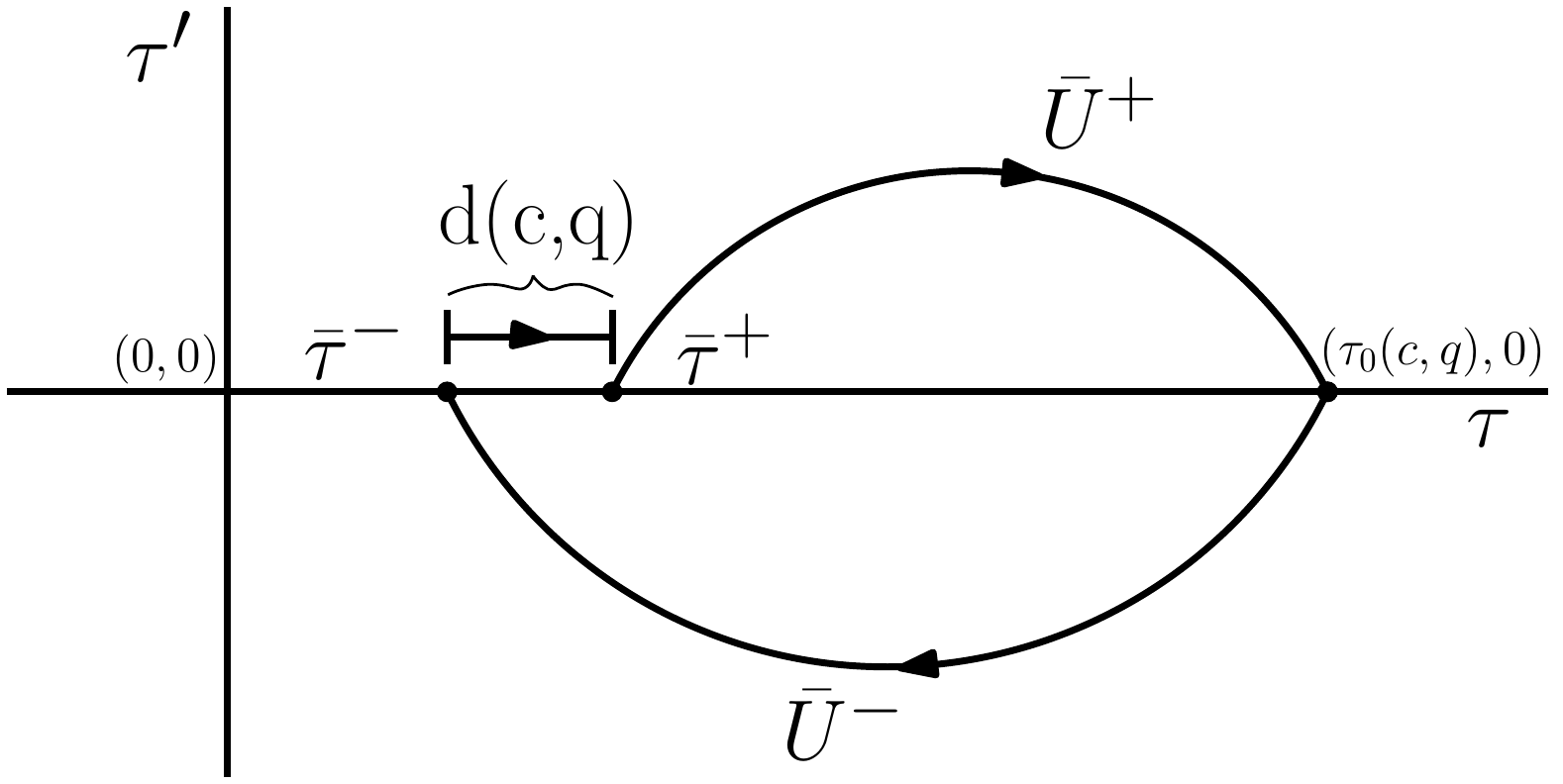}
  \caption{Depiction of the geometric interpretation of the Melnikov separation function $d(c,q)$, in the special case where $\bar U_x(0)=(0,1)^T$. Here, note that, by a direct calculation, $d(c,q)<0$.}
   \label{melnikov}
\end{center}
\end{figure}

\bl\label{l:lowfreq}
If $d(c,q)=0$, then for the linearized St. Venant eigenvalue equations \eqref{e:homlinII} about the
corresponding homoclinic orbit of \eqref{e:profile} we have $\sgn D'(0)=- \sgn \partial_c d(c,q)$.
\el

\begin{proof}
This follows by a computation similar to those of \cite{GZ,Z2,Go}.
First, notice that since $\lambda=0$ is a simple
translational eigenvalue, we have $D(0)=0$ and we may take without
loss of generality  $W_1^-=W_3^+=\bar U'$
at $\lambda=0$.  Furthermore, a matrix perturbation argument at $\lambda=0$ shows the unstable subspace
of the asymptotic matrix $A_0(\lambda)$ can be analytically extended to a neighborhood of the origin,
with limiting values at $\lambda=0$ given by the direct sum of the unstable subspace of $A_0(0)$ and
the vector $\left(1,-g_\tau(u_0,\tau_0)g_u(u_0,\tau_0)^{-1},0\right)^T$, with $g$ as in \eqref{relax}
It follows then that we can choose the vector $W_2^-(\cdot,0)$ to be asymptotically constant with
\[
V_2^-=\lim_{x\to -\infty}W_2^-(x,0)=\left(1,\frac{-(s+1)u_0}{r\tau_0},0\right)^T=\left(1,-df_*(\tau_0),0\right)^T,
\]
where the last equality follows from the relation $\tau_0^{s+1}u_0^r=1$.
Moreover, assuming the homoclinic connection goes clockwise in the $(\tau,\tau')$ variables
and that $\mu_+<0$ and $\mu_->0$ denote the decaying modes of the asymptotic matrix $A_0(0)$ at $\pm\infty$, respectively,
we have
\begin{align*}
V_1^-(0) &= \left(-\tau_0,c\tau_0,c\mu_-\tau_0^{-1}\right)^T,\\
V_3^+(0) &= \left(\tau_0,-c\tau_0,-c\mu_+\tau_0^{-1}\right)^T,
\end{align*}
where we have used the fact that $u=-c\tau$ for the vectors $W_1^-(\cdot,0)$ and $W_3^+(\cdot,0)$.  Therefore,
\begin{equation}\label{Vdet}
\begin{aligned}
\det(V_1^-,V_2^+,V_3^+)|_{\lambda=0}&=
\det\left(
       \begin{array}{ccc}
         -\tau_0 & 1 & \tau_0 \\
         c\tau_0 & -df_*(\tau_0) & -c\tau_0 \\
         c\mu_-\tau_0^{-1} & 0 & -c\mu_+\tau_0^{-1} \\
       \end{array}
    \right)\\
&=c(\mu_--\mu_+)\det\left(
                      \begin{array}{cc}
                        -1 & 1 \\
                        c & -df_*(\tau_0) \\
                      \end{array}
                    \right)\\
&=c(\mu_--\mu_+)\left(df_*(\tau_0)-c\right).
\end{aligned}
\end{equation}
which is non-zero by assumption.

Continuing, since $\det(V_1^-,V_2^-,V_3^+)$ is non-vanishing and $\det (W_1^-, W_2^-,W_3^+)|_{x=0, \lambda=0}=0$
we have
\[
D'(0)=
\frac{\partial_\lambda \det (W_1^-, W_2^-,W_3^+)|_{x=0}}
{\det(V_1^-,V_2^-,V_3^+)}|_{\lambda=0}
\]
and hence, to complete the proof then, we must calculate the numerator.  Using the Leibnitz rule, we immediately find that
\begin{equation}\label{derevans}
\begin{aligned}
\partial_\lambda \det (W_1^-, W_2^-,W_3^+)|_{\lambda=0}&=
\det (\partial_\lambda W_1^-, W_2^-,W_3^+)|_{\lambda=0} +
\dots +
\det (W_1^-, W_2^-,\partial_\lambda W_3^+)|_{\lambda=0} \\
&=
\det (\bar U_x, W_2^-,\partial_\lambda W_3^+-\partial_\lambda W_1^-)|_{\lambda=0},
\end{aligned}
\end{equation}
where the vectors $\partial_\lambda W_j|_{\lambda=0}$ satisfy at
\begin{equation}\label{jordanvec}
\partial_\lambda W_j'|_{\lambda=0}=A(x,0)\partial_\lambda W_j|_{\lambda=0}+\bar U'
\end{equation}
for $j=1,3$.  Furthermore, differentiating the stable and unstable manifold solutions of the traveling-wave
ODE at the endstate $(\tau_0,u_0)$ (chosen in equilibrium) with respect to the wave speed $c$,
we find
\[
\partial_\lambda W_1^-|_{\lambda=0}= -\left(
                                                 \begin{array}{c}
                                                   \partial_c\bar \tau^- \\
                                                   \partial_c\bar u^- \\
                                                   \bar\tau^{-2}\partial_c \bar u^- \\
                                                 \end{array}
                                               \right)
,\quad \partial_\lambda W_3^+|_{\lambda=0}=
-\left(
                                                 \begin{array}{c}
                                                   \partial_c\bar \tau^+ \\
                                                   \partial_c\bar u^+ \\
                                                   \bar\tau^{-2}\partial_c \bar u^+ \\
                                                 \end{array}
                                               \right).
\]
Note this is a reflection of the similar role of $c$ and $-\lambda$
in the structure of the linearized traveling wave equation and in the linearized eigenvalue equation, respectively.

Next, setting $\lambda=0$ in \eqref{e:homlinII} and integrating from $x=-\infty$ to $x=0$
implies
\[
(u+c\tau)|_{x=0}=c-df_*(\tau_0)
\]
for $(\tau,u)=((W_2^-)_1,(W_2^-)_2)$, the first two components of the vector $W_2^-(\cdot,0)$.  Similarly,
recalling \eqref{jordanvec} it follows that
\[
(u+c\tau)|_{x=0}=\int_{\pm\infty}^0\left(\bar{u}_x+c\bar \tau_x\right)dx=\left(\bar u+c\bar\tau\right)|_{x=\pm\infty}^0
\]
for $(\tau,u)=(\bar\tau^\pm,\bar u^{\pm})$.  Therefore, returning to \eqref{derevans}, we find that adding $\frac{1}{c}$
times the second row to the first yields
\begin{equation}
\begin{aligned}\label{Wdet}
\partial_\lambda \det (W_1^-, W_2^-,W_3^+)|_{x=0,\lambda=0}&=
\bp 0 & \frac{1}{c}\left(c-df_*(\tau_0)\right) & 0\\
\bar u_x & * & \partial_c(\bar u^--\bar u^+)\\
\bar \tau^{-2}(\bar u_x)' & * & \bar \tau^{-2}\partial_c(\bar u^--\bar u^+)'\\
\ep|_{x=0}\\
&=
\frac{1}{c}\left(df_*(\tau_0)-c\right)
\det\bp
\bar u_x  & \partial_c(\bar u^--\bar u^+)\\
\bar \tau^{-2}(\bar u_x)'  &
\bar \tau^{-2}\partial_c(\bar u^--\bar u^+)'
\ep|_{x=0}\\
&=
-c\bar \tau^{-2}(0)(df_*(\tau_0)-c)
\det\bp
\bar \tau_x  & \partial_c(\bar \tau^+-\bar \tau^-)\\
(\bar \tau_x)'  &
\partial_c(\bar \tau^+-\bar \tau^-)'
\ep|_{x=0}\\
&=
-c\bar \tau^{-2}(0)(df_*(\tau_0)-c)\partial_c d(c,q).
\end{aligned}
\end{equation}
The proof is now complete by dividing \eqref{Wdet} by \eqref{Vdet} and noting that $(\mu_--\mu_+)>0$.
\end{proof}

\bc\label{stabcor}
$\sgn \partial_c d(c,q)<0$ is necessary for stability of point spectrum of the corresponding homoclinic,
corresponding to an even number of roots with positive real part.
\ec

\br
\textup{
In the stable case considered in \cite{N2}, for which periodics appear
as $c$ is decreased from the homoclinic speed $c_{\rm hom}$,
we see that periodics exist in the situation that $d<0$, which
should be visible from the phase portrait.
That is, the lower orbit coming from the saddle point should pass
{\it outside} the upper branch.
Likewise, we can check the sign for Jin--Xin by looking at the phase
portrait for nearby periodic case.
In fact, we can compute directly that $\partial_s d>0$ for the
Jin--Xin case, as done in Section \ref{s:JXequations} below, yielding instability
for that model.
}
\er


\br\label{odestab}
\textup{
Note that $\partial_c d<0$ implies for $c<c_{\rm hom}$, where $c_{\rm hom}$ denotes the
homoclinic wavespeed, that
the unstable manifold at the saddle equilibrium spirals inward,
either toward an attracting equilibrium in which case periodics
don't exist, or toward an attracting
periodic cycle, in which case the enclosed equilibrium is a repellor.
Likewise, $\partial_c d>0$ implies for $c<c_{\rm hom}$ that
the stable manifold at the saddle equilibrium spirals inward in backward
$x$ either toward a repelling equilibrium in which case periodics
don't exist, or toward a repelling periodic cycle, in which case the
enclosed equilibrium is an attractor in forward $x$.
That is, generically, the stability condition
$\partial_c d<0$ is associated with the property
that the enclosed equilibrium be a {\it repellor},
so long as decrease in speed from $c_{\rm hom}$ is associated with
existence of periodics.
(If, rather, increase in speed is associated
with existence of periodics, then we reach the reverse conclusion that
Evans stability correspods to the property that the enclosed equilibrium
be an attractor.)
By Remark \ref{repellor}, therefore, we expect stability (of point spectra)
of homoclinics in the St. Venant case studied in \cite{N2}.
}
\er

As described in Remark \ref{odestab}, it
is a curious fact that stability in temporal variable $t$ of the homoclinic
as a solution of PDE \eqref{eqn:1conslaw}
is associated with stability in spatial variable $x$
of nearby periodics as solutions of the traveling-wave ODE \eqref{e:profile},
and, thereby,\footnote{Assuming the generic case where there is only one
limit cycle generated from increasing or decreasing the wave speed $c$,
with an associated return map that is strictly attracting or strictly
repelling.}
instability of the equilibrium contained within the periodics.
The latter is also the condition that the homoclinic itself be stable
in forward $x$-evolution as a solution of the traveling-wave ODE.
It is not clear whether this is a coincidence, or plays a deeper role
as a mechanism for stability.
However, once we understand the existence theory (itself nontrivial,
but often observed numerically), we may make conclusions readily about
stability just by examining stability or instability of the enclosed
equilibrium as a solution of the traveling-wave ODE.




\subsection{The viscous Jin-Xin equations}\label{s:JXequations}

While of theoretical interest, unfortunately the stability index $\partial_c d$ derived above appears to be analytically incomputable for the general St. Venant equations
considered here (although, one can use the ``rule of thumb" described above).
Next, we consider as simpler related example the viscous Jin--Xin model
\ba \label{eq:toy}
\tau_t - u_x&= 0,\\
u_t -c_s^2\tau_x&= -f(\tau) -u +u_{xx}
\ea
for which we can directly calculate the value of $\partial_c d$: as we will see, Lemma \ref{stabcor} holds in this simpler setting as well.  Considering traveling waves in the moving coordinate frame $x-ct$ gives traveling-wave ODE
\[
(c^2-c_s^2) \tau'=
-f(\tau) +c\tau- q -c \tau '' ,
\]
where $u=q-c\tau$.  Taking $c=c_s$, this equation is seen to be Hamiltonian and reduces to the nonlinear oscillator
\be \label{e:osc}
c \tau''= -f(\tau) +c\tau- q,
\ee
for which all periodics and the bounding homoclinic arise at once.
Indeed, this equation may be integrated via quadrature: multiplying by $\tau'$ and integrating, we obtain
\be\label{ham}
(c/2)(\tau')^2 = -F(\tau) +c\tau^2/2- q\tau +H,
\ee
$F:=\int f$,
which is of standard Hamiltonian form with ODE energy $H$ and effective potential energy $V(\tau;c,q)=F(\tau)-\frac{c}{2}\tau^2+q\tau$.
In particular, we see that the homoclinic wavespeed $c_{hom}$ is precisely $c_s$ given in \eqref{eq:toy}.

Denoting the roots of the equation $V'(\tau)=0$ as $\tau_j$, we have that orbits through the equilibrium solutions $(\tau_j,0)$
of \eqref{ham} line on the level sets $V(\tau)=H$, and correspond to homoclinic orbits provided that $V''(\tau_j)<0$.  In this case,
the orbits of \eqref{ham} can be implicitly defined via the relation
\be\label{orb}
\tau'= \pm \sqrt{ (2/c) (-F(\tau) +c\tau^2/2- q\tau +H)}
\ee
from which solutions of \eqref{eq:toy} may be obtained solving \eqref{orb} for $\tau$ as a function of $x$
and using that $u=q-c\tau$.
Furthermore, a direct calculation shows that the sign of $V''(\tau_j)=-df(\tau_j)+c$ alternates as $\tau_j$ increases,
and hence we find that the subcharacteristic condition is satisfied if and only if $V''(\tau_j)<0$.  It follows
that all non-trivial solitary wave solutions of \eqref{eq:toy} are stable to low-frequency perturbations, and hence
have stable essential spectrum.

%
%

%
%
%

The analysis of the point spectrum of the corresponding homoclinic orbits
is handled by the usual Evans function framework as described in Section \ref{s:evans_framework}.
To this end, notice that in co-moving coordinates the linearized equations are
\ba \label{eqn:cotoy}
\tau_t-c\tau_x - u_x&= 0,\\
u_t-cu_x  -c_s^2\tau_x&=
-df(\bar \tau) \tau - u +u_{xx} ,
\ea
where $\bar U=(\bar \tau, \bar u)$ denotes background profile.  Taking the Laplace transform in time
immediately leads to the linearized eigenvalue ODE
\ba \label{eqn:teig}
\lambda \tau-c\tau' - u'&= 0,\\
\lambda u-cu'  -c_s^2\tau'&=
-df(\bar \tau) \tau - u +u'' ,
\ea
%
governing the spectral stability of the underlying wave to localized perturbations.
In order to use the abstract Evans function setup of Section \ref{s:evans_framework},
we write the eigenvalue system \eqref{eqn:teig} as a first-order system of the form
$W'=A(x,\lambda)W$, where
\be
W:=\bp \tau \\ u\\u'\ep,
\qquad
A:=\bp
\lambda/c & 0 & -1/c\\
0 & 0 & 1\\
-c_s^2\lambda/c +df(\bar \tau)
& \lambda + 1  & c_s^2/c-c\\
\ep.
\ee
In order to apply Lemma \ref{limitingsystem}, and hence define the corresponding Evans function $D(\lambda)$
as in \eqref{eq:evansdef}, we must verify that we have extended consistent splitting in this case.


To obtain consistent splitting in the homoclinic case,
we must have stability of constant solutions at the endpoint $\tau=0$ from
which the homoclinic originates,
which as noted earlier is equivalent to
the subcharacteristic condition
$$
df_*(\tau_j)<c_s,
$$
where $f_*(\tau):=-u_*(\tau)= f(\tau)$
is determined by the value
$u_*(\tau)= -f(u) $ at equilibrium,
hence
$df_*(\tau)=df(\tau)$.
Thus, the subcharacteristic condition is
$df(\tau)<c_s$, and so as discussed above it is
satisfied at alternate roots.
(Recall that, by our previous analysis, instability
of the constant solution is necessary for Hopf bifurcation
and thus for existence of a nonlinear center and
bounding homoclinic.)
With this setup, we see that Lemma \ref{limitingsystem} applies to the eigenvalue system
corresponding to \eqref{eqn:teig} and hence we may define the Evans function $D(\lambda)$
as in \eqref{eq:evansdef}.  Furthermore, a direct calculation shows that the proof of Corollary \ref{stabcor}
carries over line by line to the viscous Jin--Xin case considered here.  It follows then that
we can determine the stability of the solitary wave by direct considerations of the Melnikov separation
function $d(c,q)$ defined in \eqref{mel}.  It turns out that due to the geometric nature of the function
$d(c,q)$ the sign of the derivative $\partial_c d$ is trivial to compute at the homoclinic.


Indeed, note that for general $c$ the traveling-wave ODE is
\be \label{e:osc2}
c \tau''= -f(\tau) +c\tau- q - (c^2-c_s^2)\tau'.
\ee
It follows that $H'= -(c^2-c_s^2)|\tau'|^2$, where $H$ is the the Hamiltonian defined as in \eqref{ham},
and hence, for $c>c_s$, the ODE flow decreases the Hamiltonian $H$.
Noting that $H$ is minimized at the equilibrium enclosed by the Homoclinic, it follows that we must
have $d>0$ for $0<c-c_s\ll 1$ and hence we conclude that $\partial_c d>0$.  By Corollary \ref{stabcor} then,
we conclude that the solitary waves of \eqref{eq:toy} always have unstable point spectrum.

This example, although much simpler than the full St. Venant equations, illustrates the power of being able to
relate the stability of the solitary wave as a solution of the governing PDE to information about Melnikov separation
function $d(c,q)$, an inherently geometric quantity encoding information about the ODE phase space.
For the St. Venant equations, however, we must resort to numerics to study the point spectrum of the linearized
operator.  This analysis is carred out in the following two sections.  As we will see, however, although we are not
able to analytically compute the stability index $\sgn\partial_c d(c,q)$ in this general case, our
numerical experiments are consistent with the rule of thumb relating
the PDE  stability to the stability of the enclosed equilibrium solution of the ODE.

\section{High-frequency asymptotics}\label{s:track}

Our next goal is to carry out a numerical Evans study of the point spectrum
for the linearized equations \eqref{e:homlinII} for homoclinic orbits
corresponding to various choices of the turbulent friction parameters $(r,s)$.
Indeed, later we will see that a winding number argument allows us to numerically compute the number of
roots of the Evans function contained in a given compact region of the complex plane.
The goal of this section is thus to eliminate the possibility of arbitrarily large
unstable eigenvalues, so as to reduce our computations to
a compact domain in $\lambda$ \cite{Br,BrZ,HLZ,HLyZ1,HLyZ2}.
This can be done by a number of techniques, including energy
estimates and high-frequency asymptotics.
As energy estimates degrade with small viscosity, we opt to
use the latter techniques, which is well adapted to the situation
of multiple scales such as are involved in our examples \cite{PZ,Z5}.

\subsection{Approximate block-diagonalization}\label{approxdiagsubsec}

To begin, we show that the coefficient matrix $A(x,\lambda)$ in the Evans system \eqref{homAII}
can be written, for $|\lambda|\gg 1$, as a block diagonal matrix plus an
asymptotically small term.  We begin by expanding $W'(\cdot)=A(\cdot,\lambda)W(\cdot)$ as
$A=B+R$, where
\ba\label{BR}
B&:=\bp
\frac{\lambda}{c}& 0 & 0\\
0 & 0 &  \bar \tau^2\\
\frac{ -\bar \alpha \lambda}{c\nu} &
\frac{\lambda }{\nu} & 0
\ep,
\\
R&:=\bp
0& 0 & -\frac{\bar \tau^2}{c}\\
0 & 0 &  0\\
\frac{(s+1)\bar \tau^s (q-c\bar \tau)^r -\bar \alpha_x }{\nu} &
\frac{ r\bar \tau^{s+1}(q-c\bar \tau)^{r-1}}{\nu} &
\frac{-c\bar \tau^2+\bar \alpha \bar \tau^2/c}{\nu}
\ep
=:
\bp
0& 0 & k\\
0 & 0 &  0\\
l& m & n
\ep,
\ea
where $\bar \alpha$ and $\bar \alpha_x$ are defined in \eqref{alphaII}
and where $\tau''$ 
is then computed using equation \eqref{e:profile}. Since $R=O(1)$ as $|\lambda|\to \infty$, we expect that the behavior of the Evans system \eqref{homAII} is governed by the principle part $B(x,\lambda)$ for $\lambda$ sufficiently large.  Notice,
however, that it is not straightforward to characterize this due to the multiple spectral scales associated with the growth
of the eigenvalues of $B$.  Indeed, a straightforward computation shows that the eigenvalues of the principle matrix
$B$ are given by $\frac{\lambda}{c}$ and $\pm\bar \tau \sqrt{\frac{\lambda}{\nu}}$, and hence the spectrum
of $B$ has two principle growth rates: order $\lambda$ and order $\lambda^{1/2}$.  In order to keep track of both of these
scales, one must make a number of careful transformations, which we now describe in detail.



First, we introduce the transformations
\be\label{T}
T:=\bp 1 & 0 &0\\
0&1&0\\
\theta&0 & 1\ep,
\qquad
T^{-1}=\bp 1 & 0 &0\\
0&1&0\\
-\theta & 0&0\ep,
\ee
where $\theta:=-\frac{\bar \alpha}{\nu}$.
Then we readily see that
\be
T^{-1}BT=
\bp\frac{\lambda}{c}& 0 & 0\\
\theta \bar\tau^2 & 0 &  \bar \tau^2\\
0 &\frac{\lambda }{\nu} & 0\ep,\qquad
T^{-1}RT=
\bp k\theta& 0 & k\\
0& 0 & 0\\
l+n\theta-k\theta^2 & m& n-k\theta\ep
\ee
and, defining
\be\label{tilde B}
B_1:= T^{-1}BT -\bp 0 & 0 & 0\\ \bar \tau^2\theta&0& 0\\
0&0&0\ep=
\bp
\lambda/c& 0 & 0\\
0 & 0 &  \bar \tau^2\\
0&\frac{\lambda }{\nu} & 0
\ep,
\qquad W:=TW_1,
\ee
we have $W_1'=(B_1+R_1)W_1$, where
\ba\label{R1}
R_1&:= T^{-1}RT - T^{-1}T'
+\bp 0 & 0 & 0\\ \bar \tau^2\theta&0& 0\\0&0&0\ep\\
&=
\bp k\theta& 0 & k\\
\bar\tau^2\theta& 0 & 0\\
l+n\theta-k\theta^2 -\theta'& m& n-k\theta\ep
=:
\bp
k_1 & 0 &  l_1\\
m_1& 0 &0\\
n_1&o_1&p_1
\ep.
\ea

Now, define
\ba\label{T_1}
T_1&:=\bp
1&0 & 0\\
0&1 & 0\\
0&0 & \mu\sqrt{\lambda}\\
\ep,
\qquad
T_1^{-1}&=\bp
1&0 & 0\\
0&1 & 0\\
0&0 & (\mu\sqrt{\lambda})^{-1}\\
\ep,
\ea
where $\mu=\frac{1}{ \bar \tau \sqrt{\nu}}$ and notice that
\be
T_1^{-1}B_1T_1=
\bp\frac{\lambda}{c} & 0 &  0\\
0& 0 &  \sqrt{\lambda}\mu\bar\tau^2\\
0& \sqrt{\lambda}\mu\bar\tau^2 &  0\ep,\qquad
T_1^{-1}R_1T_1=
\bp k_1&0&l_1\mu\sqrt{\lambda}\\
m_1&0&0\\
\frac{n_1}{\mu\sqrt{\lambda}}&\frac{o_1}{\mu\sqrt{\lambda}}&p_1
\ep.
\ee
Setting $W_1=T_1W_2$, we thus have
$W_2'=(B_2+R_2)W_2$, where
\ba\label{B2}
B_2:=
\bp
\frac{\lambda}{c} & 0 &  l_1\mu\sqrt{\lambda}\\
0& 0 &  \sqrt{\lambda}\mu\bar\tau^2\\
0& \sqrt{\lambda}\mu\bar\tau^2 &  0\\
\ep
\ea
and
\ba\label{R2}
R_2:= T_1^{-1}R_1T_1 - T_1^{-1}T_1'-\bp0&0&l_1\mu\sqrt{\lambda}\\0&0&0\\0&0&0\ep
&=
\bp
k_1 & 0 &  0\\
m_1& o_1 &0\\
\frac{n_1}{\sqrt{\lambda}\mu}&\frac{o_1}{\sqrt{\lambda}\mu}&p_2
\ep,
\ea
with $p_2:=p_1- \frac{\mu'}{\mu}=
p_1+ \frac{\bar \tau'}{\bar \tau}$.

Continuing, define
\ba\label{T_2}
T_2&:=\bp
1&0 & 0\\
0&1 & -1\\
0&1 & 1\\
\ep,
\qquad
T_1^{-1}&=\bp
1&0 & 0\\
0&\frac{1}{2} &\frac{1}{2}\\
0&-\frac{1}{2} & \frac{1}{2}\\
\ep
\ea
and set $B_3:= T_2^{-1}B_2T_2$ and $W_2:=T_2W_3$.  Then by a direct calculation we have
\be\label{finaltrack}
W_3'=(B_3+R_3)W_3,
\ee
where
\ba\label{B3}
B_3= \bp \frac{\lambda}{c} & l_1\mu\sqrt{\lambda}& l_1\mu\sqrt{\lambda}\\
0&  \sqrt{\lambda}\mu \bar\tau^2 &0\\
0& 0&  -\sqrt{\lambda}\mu \bar\tau^2
\ep
=\bp\frac{\lambda}{c}&-\frac{\bar\tau}{c}\sqrt{\frac{\lambda}{\nu}}&-\frac{\bar\tau}{c}\sqrt{\frac{\lambda}{\nu}}\\
0&  \bar\tau\sqrt{\frac{\lambda}{\nu}} &0\\
0&  0 & -\bar\tau\sqrt{\frac{\lambda}{\nu}} \\
\ep
\ea
and
\ba\label{R3}
R_3&:= T_2^{-1}R_2T_2 - T_2^{-1}T_2'\\
&=
\bp
k_1 & 0 &  0\\
\frac{m_1+n_1/\sqrt{\lambda}\mu}{2}& \frac{p_2+o_1/\sqrt{\lambda}\mu}{2}
& \frac{p_2-o_1/\sqrt{\lambda}\mu}{2}\\
\frac{-m_1+n_1/\sqrt{\lambda}\mu}{2}& \frac{p_2+o_1/\sqrt{\lambda}\mu}{2}&
\frac{p_2-o_1/\sqrt{\lambda}\mu}{2}\\
\ep
=:\bp k_3&0&0\\
l_3&m_3&n_3\\
o_3&m_3&n_3\ep.
\ea

Finally, we introduce the transformations
\ba
T_3&:=\bp1&\frac{\tilde{\theta}}{\sqrt{\lambda}}&\frac{\tilde{\theta}}{\sqrt{\lambda}}\\
0&1&0\\
0&0&1
\ep,\quad
T_3^{-1}&=\bp1&-\frac{\tilde{\theta}}{\sqrt{\lambda}}&-\frac{\tilde{\theta}}{\sqrt{\lambda}}\\
0&1&0\\
0&0&1
\ep,
\ea
where $\tilde{\theta}=\frac{\bar\tau}{\sqrt{\nu}}$ and note that
\ba
&T_3^{-1}B_3 T_3=\bp\frac{\lambda}{c}&-\frac{\bar\tau^2}{\nu}&\frac{\bar\tau^2}{\nu}\\
0&  \bar\tau\sqrt{\frac{\lambda}{\nu}} &0\\
0&  0 & -\bar\tau\sqrt{\frac{\lambda}{\nu}} \\
\ep,\\
&T_3^{-1}R_3 T_3=\\
&\bp k_3-(l_3+o_3)\frac{\tilde\theta}{\sqrt{\lambda}}
&(k_3-2m_3-(l_3+o_3)\frac{\tilde\theta}{\sqrt{\lambda}})\frac{\tilde\theta}{\sqrt{\lambda}}
&(k_3-2n_3-(l_3+o_3)\frac{\tilde\theta}{\sqrt{\lambda}})\frac{\tilde\theta}{\sqrt{\lambda}}\\
l_3&m_3+l_3\frac{\tilde\theta}{\sqrt{\lambda}}&n_3+l_3\frac{\tilde\theta}{\sqrt{\lambda}}\\
o_3&m_3+o_3\frac{\tilde\theta}{\sqrt{\lambda}}&n_3+o_3\frac{\tilde\theta}{\sqrt{\lambda}}\ep.
\ea
Defining then
\be
B_4\ =\ \bp\frac{\lambda}{c}&0&0\\
0&  \bar\tau\sqrt{\frac{\lambda}{\nu}} &0\\
0&  0 & -\bar\tau\sqrt{\frac{\lambda}{\nu}} \\
\ep,
\ee
and
\be
R_4\ =\ T_3^{-1}R_3 T_3+\bp0&-\frac{\bar\tau^2}{\nu}-\frac{\tilde\theta'}{\sqrt{\lambda}}
&\frac{\bar\tau^2}{\nu}-\frac{\tilde\theta'}{\sqrt{\lambda}}\\
0&0&0\\
0&0&0\ep,
\ee
it follows that the original Evans system \eqref{homAII} can be written in the approximately block-diagonal
system
\begin{equation}\label{approxdiag1}
W'=\left(B_4(x,\lambda)+R_4(x,\lambda)\right)W
\end{equation}
where $R_4=o(\lambda^\eps)$ as $\lambda\to \infty$ for all $\eps>0$.  With this result in hand, we can now
obtain bounds on the size of the unstable eigenvalues $\lambda$ corresponding to the linearized
St. Venant equations \eqref{e:homlinII}.  Before stating these bounds, however, we need the following
abstract tracking lemma concerning the eigenspaces of approximately block-diagonal systems.


\subsection{Quantitative tracking lemma}\label{quant}

We now recall the following quantitative estimate \cite{HLyZ1,HLyZ2}.
Consider an asymptotically constant approximately block-diagonal system
$W'=(M+\Theta)(x)W$,
\begin{equation}\label{tracksys}
 W=\begin{pmatrix}W_-\\ W_+\end{pmatrix}, \quad
M= \begin{pmatrix}M_-& 0\\ 0 & M_+ \end{pmatrix},
\quad
\Theta= \begin{pmatrix}\Theta_{--}& \Theta_{-+}\\
\Theta_{+-} & \Theta_{++} \end{pmatrix},
\end{equation}
with $\Re M_-\le c_-I$ and $\Re M_+\ge c_+I$ satisfying
\begin{equation}\label{numrange}
c_+-c_- \ge \delta(x)>0,
\end{equation}
where $\Re M:=(1/2)(M+M^*)$ denotes the symmetric part of a matrix $M$.

Denote by
\begin{equation}\label{zeta+-}
\begin{aligned}
\zeta_\pm(x)&:=
\frac{ \delta -|\Theta_{--}|-|\Theta_{++}|}{2|\Theta_{+-}|}
\pm
\sqrt{
\Big(\frac{ \delta -|\Theta_{--}|-|\Theta_{++}|}{2|\Theta_{+-}|}\Big)^2
-\frac{|\Theta_{-+}|}{|\Theta_{+-}|} }
\end{aligned}
\end{equation}
the roots of
\begin{equation}\label{quad}
P(\delta,x):=
\Big(-\delta +|\Theta_{--}|+|\Theta_{++}|\Big)\zeta
+|\Theta_{-+}|+ |\Theta_{+-}|\zeta^2=0,
\end{equation}
where, here and below, $|\cdot|$ denotes the $\ell^2$ matrix
operator norm.

\begin{lemma}[\cite{HLyZ1,HLyZ2}]\label{tracklem}
Suppose that
\begin{equation}\label{rootcond}
\delta> |\Theta_{--}|+|\Theta_{++}| + 2\sqrt{|\Theta_{-+}||\Theta_{+-}|}.
\end{equation}
Then, (i) $0<\zeta_-<\zeta_+$,
(ii) the invariant subspaces of the limiting coefficient
matrices $(M+\Theta)(\pm \infty)$ are contained in distinct cones
$$
\Omega_-=\{|W_-|/|W_+|\le \zeta_-\},
\quad
\Omega_+=\{|W_-|/|W_+|\ge \zeta_+\},
$$
and
(iii) denoting by $S^+$ the total eigenspace of $(M+\Theta)(+\infty)$
contained in $\Omega_+(+\infty)$
and $U^-$ the total eigenspace of $(M+\Theta)(-\infty)$
contained in $\Omega_-(-\infty)$,
the manifolds of solutions of \eqref{tracksys}
asymptotic to $S^+$ at $x=+\infty$ $U^-$ at $x=-\infty$
are separated for all $x$, lying in $\Omega_+$ and $\Omega_-$ respectively.
In particular, there exist no solutions of \eqref{tracksys} asymptotic
to $S^+$ at $+\infty$ and to $U^-$ at $- \infty$.
\end{lemma}

\br\label{noevans}
\textup{
In the case that $c_-<0$ at $+\infty$ and $c_+>0$ at $x=-\infty $,
$S^+$ and $U^-$ correspond to
the stable subspace at $+\infty$ and the unstable subspace at
$-\infty$ of the limiting constant-coefficient matrices, and
we may conclude nonexistence of decaying solutions of \eqref{tracksys},
or nonvanishing of the associated Evans function.
More generally, even if (as here), the property of decay is lost as
we traverse the boundary of consistent splitting, so long as there remains
a positive spectral gap $\delta$, and \eqref{rootcond} remains satisfied,
we may still conclude nonvanishing of the Evans function defined as above
by continuous extension of subspaces $U^-$ and $S^+$.
}
\er

\begin{proof}(\cite{HLyZ1,HLyZ2})
From \eqref{tracksys}, we obtain readily
\begin{equation}
\begin{aligned}
|W_-|'&\le  c_-|W_-|+
 |\Theta_{--}||W_-|+ |\Theta_{-+}||W_+|,\\
|W_+|'&\ge
c_+|W_+|
-|\Theta_{+-}||W_-| - |\Theta_{++}||W_+|,\\
\end{aligned}
\end{equation}
from which, defining $\zeta:= |W_-|/|W_+|$,
we obtain by a straightforward computation the Riccati equation
$\zeta'\le P(\zeta,x)$.

Consulting \eqref{quad}, we see that $\zeta'<0$ on the interval
$\zeta_-<\zeta<\zeta_+$,
whence $\Omega_-:=\{\zeta\le \zeta_-\}$ is an invariant region
under the forward flow of \eqref{tracksys}; moreover, this region
is exponentially attracting for $\zeta < \zeta_+$.
A symmetric argument yields that $\Omega_+:=\{\zeta \ge \zeta_+\}$ is
invariant under the backward flow of \eqref{tracksys}, and exponentially
attracting for $\zeta >\zeta_-$.
Specializing these observations to the constant-coefficient limiting
systems at $x=-\infty$ and $x=+\infty$, we find that the invariant
subspaces of the limiting coefficient matrices must lie in
$\Omega_-$ or $\Omega_+$.  (This is immediate in the diagonalizable
case; the general case follows by a limiting argument.)

By  forward (resp. backward) invariance of $\Omega_-$ (resp. $\Omega_+$),
under the full, variable-coefficient flow, we thus find that the manifold
of solutions initiated along $U^+$ at $x=-\infty$
lies for all $x$ in $\Omega_+$ while the manifold of
solutions initiated in $S^+$ at $x=+\infty$ lies for all $x$ in $\Omega_-$.
Since $\Omega_-$ and $\Omega_+$ are distinct,
we may conclude that under condition \eqref{rootcond}
there are no solutions asymptotic to both $U^-$ and $S^+$.
\end{proof}

\subsection{Bounds on unstable eigenvalues}
With the above tracking lemma in hand, we are able to state our main result
concerning the confinement of the unstable eigenvalues of \eqref{e:homlinII}.
After a change of coordinates switching $-$ and $+$, using notations of Section \ref{approxdiagsubsec},
equation \eqref{approxdiag1} determines a system of form
\eqref{tracksys} with
\ba \label{spectheta}
M_{+}&= \bp \frac{\lambda}{c} & 0 \\
0&  \bar\tau\sqrt{\frac{\lambda}{\nu}} \\ \ep,
\quad
M_{-} = \bp -\bar\tau\sqrt{\frac{\lambda}{\nu}}  \ep\\
\Theta_{++}&=
\Theta^0_{++}+\lambda^{-1/2}\Theta^1_{++}
+\lambda^{-1}\Theta^2_{++}+\lambda^{-3/2}\Theta^3_{++}\\
&:=
\bp
k_1 & -\frac{\bar\tau^2}{\nu} \\
\frac{m_1}{2}&\frac{p_2}{2}\\
\ep
+
\lambda^{-1/2}
\bp 0 & -\tilde\theta'+\tilde\theta(k_1-p_2)\\
\frac{n_1}{2\mu}&\frac{o_1}{2\mu}-\tilde\theta\frac{m_1}{2}\ep
\\
&\quad+\lambda^{-1}
\bp -\tilde\theta\frac{n_1}{\mu}&-\tilde\theta\frac{o_1}{\mu}\\
0&\tilde\theta\frac{n_1}{2\mu}\ep
+\lambda^{-3/2}
\bp 0&-\tilde\theta^2\frac{n_1}{\mu}\\
0&0\ep,\\
\\
\Theta_{+-}&=
\Theta^0_{+-} +\lambda^{-1/2}\Theta^1_{+-}
+\lambda^{-1}\Theta^2_{+-}+\lambda^{-3/2}\Theta^3_{+-}\\
&:=
\bp \frac{\bar\tau^2}{\nu}\\ \frac{p_2}{2}\\ \ep
+\lambda^{-1/2}
\bp -\tilde\theta'+\tilde\theta(k_1-p_2)\\ -\frac{o_1}{2\mu}+\tilde\theta\frac{m_1}{2}\\ \ep
+\lambda^{-1}
\bp\tilde\theta\frac{o_1}{\mu} \\ \tilde\theta\frac{n_1}{2\mu}\ep
+\lambda^{-3/2}
\bp-\tilde\theta^2\frac{n_1}{\mu} \\ 0\ep,\\
\Theta_{-+}&=
\Theta^0_{-+} + \lambda^{-1/2}\Theta^1_{-+}
+\lambda^{-1}\Theta^2_{-+}\\
&:=
\bp -\frac{m_1}{2}&\frac{p_2}{2} \ep
+
\lambda^{-1/2}
\bp \frac{n_1}{2\mu}& \frac{o_1}{2\mu}-\tilde\theta\frac{m_1}{2} \ep
+
\lambda^{-1}
\bp 0& \tilde\theta\frac{n_1}{2\mu}\ep
,\\
\Theta_{--}&=
\Theta^0_{--} + \lambda^{-1/2}\Theta^1_{--}
+ \lambda^{-1}\Theta^2_{--} \\
&:= \bp \frac{p_2}{2}\\ \ep
+\lambda^{-1/2} \bp -\frac{o_1}{2\mu}-\tilde\theta\frac{m_1}{2}\ep
+\lambda^{-1}\bp\tilde\theta\frac{n_1}{2\mu}\ep,
\ea
for which evidently
$\delta \ge \bar\tau\Re \sqrt{\frac{\lambda}{\nu}}
\ge |\lambda|^{1/2}\frac{\bar \tau}{\sqrt{2\nu}}$
when $\Re \lambda \ge 0$.  Then a direct application of Lemma \ref{tracklem} yields the
following high-frequency bound.

\bc\label{HFcorollary}
There are no unstable roots $\lambda$, with $\Re \lambda \ge 0$,
of $D$ with $|\lambda|>\max_xR^4$, where $R$ is the only positive root of $X^8-a_0X^6-a_{1/2}X^5-a_1X^4-a_{3/2}X^3-a_2X^2-a_{5/4}X-a_3$ with
\ba\label{ab}
a_0&=\frac{\sqrt{2\nu}\left(|\Theta^0_{--}|+|\Theta^0_{++}| + 2\sqrt{|\Theta^0_{-+}||\Theta^0_{+-}|}\right)}{\bar \tau},\\
a_{1/2}&:=\frac{\sqrt{2\nu}2\sqrt{|\Theta^0_{-+}||\Theta^1_{+-}|+|\Theta^1_{-+}||\Theta^0_{+-}|}}{\bar\tau},\\
a_1&:=\frac{\sqrt{2\nu}\left(|\Theta^1_{--}|+|\Theta^1_{++}| + 2\sqrt{|\Theta^1_{-+}||\Theta^1_{+-}|+|\Theta^0_{-+}||\Theta^2_{+-}|+|\Theta^2_{-+}||\Theta^0_{+-}|}\right)}{ \bar \tau},\\
a_{3/2}&:=\frac{\sqrt{2\nu}2\sqrt{|\Theta^0_{-+}||\Theta^3_{+-}|+|\Theta^1_{-+}||\Theta^2_{+-}|+|\Theta^2_{-+}||\Theta^1_{+-}|}}{\bar\tau},\\
a_2&:=\frac{\sqrt{2\nu}\left(|\Theta^2_{--}|+|\Theta^2_{++}| + 2\sqrt{|\Theta^2_{-+}||\Theta^2_{+-}|+|\Theta^1_{-+}||\Theta^3_{+-}|}\right)}{\bar \tau},\\
a_{5/4}&:=\frac{\sqrt{2\nu}2\sqrt{|\Theta^2_{-+}||\Theta^3_{+-}|}}{\bar\tau},\\
a_3&:=\frac{\sqrt{2\nu}|\Theta^3_{++}|}{\bar \tau}.
\ea
In particular there are no unstable roots $\lambda$ with modulus larger than $\tilde R^4$ where $\tilde R$ is the only
positive root of $X^8-(\max a_0)X^6-(\max a_{1/2})X^5-(\max a_1)X^4-(\max a_{3/2})X^3-(\max a_2)X^2-(\max a_{5/4})X-\max a_3$.
\ec

To aid in the utilization of the bounds in Corollary \ref{HFcorollary}, note that $R^4$ is smaller than
\begin{itemize}
\item the square of the only positive root of
\ba
X^4-\left(a_0X+\frac{a_{1/2}}{2}\right)X^3-\left(\frac{a_{1/2}}{2}+a_1+\frac{a_{3/2}}{2}\right)X^2-\left(\frac{a_{3/2}}{2}+a_2+\frac{a_{5/4}}{2}\right)X-\left(\frac{a_{5/4}}{2}+a_3\right)\\
=:X^4-\tilde a_0 X^3-\tilde a_1 X^2- \tilde a_2 X-\tilde a_3;
\ea
\item the only positive root of
\ba
X^2-(\tilde a_0^2+2\tilde a_1+\tilde a_2)X-(\tilde a_2+2\tilde a_3).
\ea
\end{itemize}
Both bounds can be explicitly given.  Here, however, we determine $R$ by numerical solution of the full eighth-order
polynomial equation given in Corollary \eqref{HFcorollary}.

\begin{proof}
By Lemma \ref{tracklem}, the triangle inequality and subadditivity of the square-root function,
there are no unstable eigenvalues $\lambda$ with
\ba\label{Rcomp}
|\lambda|^{1/2}&>
\frac{\sqrt{2\nu}\left(|\Theta^0_{--}|+|\Theta^0_{++}| + 2\sqrt{|\Theta^0_{-+}||\Theta^0_{+-}|}\right)}{\bar \tau}\\
&\quad
+|\lambda|^{-1/4}
\frac{\sqrt{2\nu}2\sqrt{|\Theta^0_{-+}||\Theta^1_{+-}|+|\Theta^1_{-+}||\Theta^0_{+-}|}}{\bar\tau}\\
&\quad
+|\lambda|^{-1/2}
\frac{\sqrt{2\nu}\left(|\Theta^1_{--}|+|\Theta^1_{++}| + 2\sqrt{|\Theta^1_{-+}||\Theta^1_{+-}|+|\Theta^0_{-+}||\Theta^2_{+-}|+|\Theta^2_{-+}||\Theta^0_{+-}|}\right)}{ \bar \tau}\\
&\quad
+|\lambda|^{-3/4}
\frac{\sqrt{2\nu}2\sqrt{|\Theta^0_{-+}||\Theta^3_{+-}|+|\Theta^1_{-+}||\Theta^2_{+-}|+|\Theta^2_{-+}||\Theta^1_{+-}|}}{\bar\tau}\\
&\quad
+|\lambda|^{-1}
\frac{\sqrt{2\nu}\left(|\Theta^2_{--}|+|\Theta^2_{++}| + 2\sqrt{|\Theta^2_{-+}||\Theta^2_{+-}|+|\Theta^1_{-+}||\Theta^3_{+-}|}\right)}{\bar \tau}\\
&\quad
+|\lambda|^{-5/4}
\frac{\sqrt{2\nu}2\sqrt{|\Theta^2_{-+}||\Theta^3_{+-}|}}{\bar\tau}\\
&\quad
+|\lambda|^{-3/2}
\frac{\sqrt{2\nu}|\Theta^3_{++}|}{\bar \tau},
\ea
also written $|\lambda|^{1/2}>a_0+a_{1/2}|\lambda|^{-1/4}+a_1|\lambda|^{-1/2}+a_{3/2}|\lambda|^{-3/4}+a_2|\lambda|^{-1}+a_{5/4}|\lambda|^{-5/4}+a_3|\lambda|^{-3/2}$.
The condition is satisfied as soon as $|\lambda|^{1/4}>\max_x R$.
\end{proof}

For those interested in reproducing or using the high frequency bounds obtained above, we record in the Appendix
the formulas for the various terms listed above in terms of the original underlying homoclinic profile.

\br
By similar estimates (see the abstract Tracking Lemma of \cite{MaZ3,PZ,Z5}),
one may obtain as in \cite{HLyZ1} the asymptotic description
\begin{equation}\label{conv}
D(\lambda)\sim c_1 e^{c_2\sqrt{\lambda}}
\end{equation}
for some $c_1,c_2\in\RM$ as $|\lambda|\to\infty$, with convergence at rate $\mathcal{O}(|\lambda|^{-1/2})$.
\er

\section{Numerical investigations}\label{s:num}

In this section, we carry out a numerical study of the stability of the solitary wave solutions of the generalized St. Venant equations.
In particular, we perform numerical Evans function calculations to detect the presence of unstable point spectrum.  In the two
examples considered, we find that one has stable essential spectrum with unstable point spectrum, while the other has stable point
spectrum with unstable essential spectrum.  Furthermore, for the latter case we perform a time evolution study to illustrate
the convective nature of the instability arising from the essential spectrum.  For illustrative purposes, however, we
choose to perform a numerical Evans study of the solitary wave solutions of the viscous Jin--Xin equations.  Recall
that by our considerations in Section \ref{s:JXequations}, we expect the linearization to have unstable point spectrum, in particular,
a positive real eigenvalue, corresponding to exponential instability of the solitary wave.

\subsection{Jin--Xin Example}

To begin, consider first the set of equations
\ba \label{eq:toy2}
\tau_t - u_x&= 0,\\
u_t -\tau_x&= \tau^{-1/2} -u +u_{xx},
\ea
corresponding to \eqref{eq:toy} with nonlinearity $f(\tau)=-\tau^{-1/2}$, mimicking the St. Venant structure, and $c_s=1$.
For this example, we consider the solitary wave profile $\bar U=(\bar \tau,\bar u)$ satisfying the ODE
\[
\frac{1}{2}\left(\bar\tau'\right)^2=\frac{1}{3}\bar\tau^3+\frac{1}{2}\bar\tau^2-2\bar\tau-\frac{1}{2}
\]
corresponding to the profile ODE \eqref{ham} with $(q,H)=\left(2,-\frac{1}{2}\right)$.  This profile is numerically generated using MATLAB's boundary-value solver bvp5c \cite{KL} and is depicted in phase space Figure \ref{JinXin_Evans_2}(a) and is plotted against the spatial variable $x$ in Figure \ref{JinXin_Evans_2}(b).   To study
the spectral stability of this profile, we must next consider the linearized eigenvalue equations \eqref{eqn:teig}
with $c_s=1$ and $df(\bar\tau)=-2\bar \tau$.  Since we expect the existence of a real unstable eigenvalue by the
calculations in Section \ref{s:JXequations},
we numerically compute the Evans function along the positive real line and show the existence of a positive root.  This is accomplished
here by using the polar-coordinate method of \cite{HuZ} as well as other standard procedures; see also \cite{BrZ}.
In particular, the Evans function was computed in MATLAB using RK45 to solve the associated ODE's.
The adaptive error control provided by RK45 gives an estimate on truncation error which, by the results of \cite{Z6}, then translate
to convergence error estimates. Furthermore, we find analytically varying initializing bases $\{R_j^{\pm}\}$ for the Evans function computation
via the method of Kato; see \cite{GZ,HuZ,BrZ,BHZ}.  Following this procedure then, we plot in Figure \ref{JinXin_Evans_2} (c)
the output as evaluated on the interval $[0,1]$ and clearly see that $D(\lambda)$ vanishes (geometrically) twice within this interval: once at
$\lambda=0$ corresponding to the simple translational eigenvalue and once at a positive $\lambda_0>0$ yielding the
expected instability.
\begin{figure}[htbp]
\begin{center}
$
\begin{array}{lr}
  (a)\includegraphics[scale=0.25]{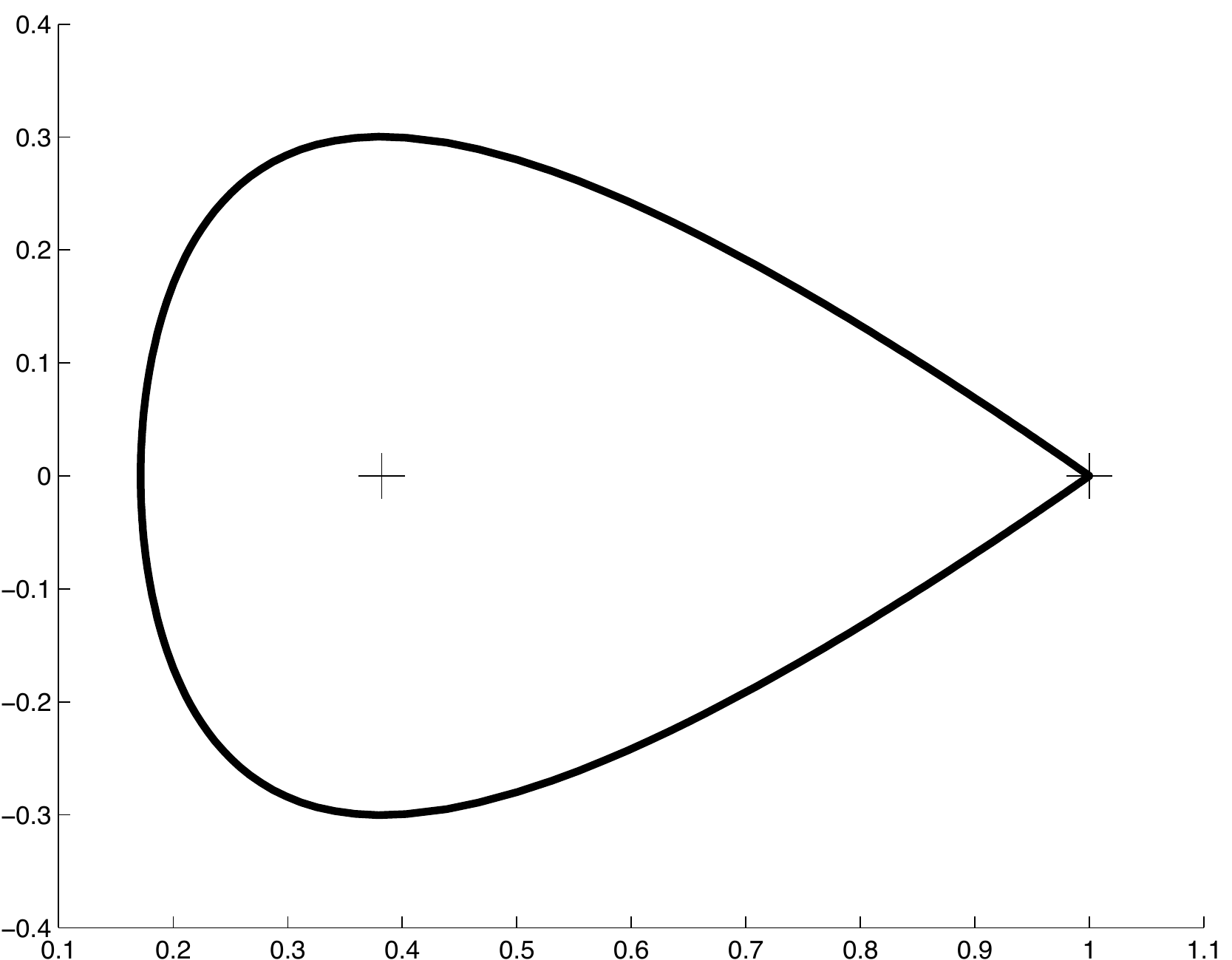}   \quad  (b)  \includegraphics[scale=.25]{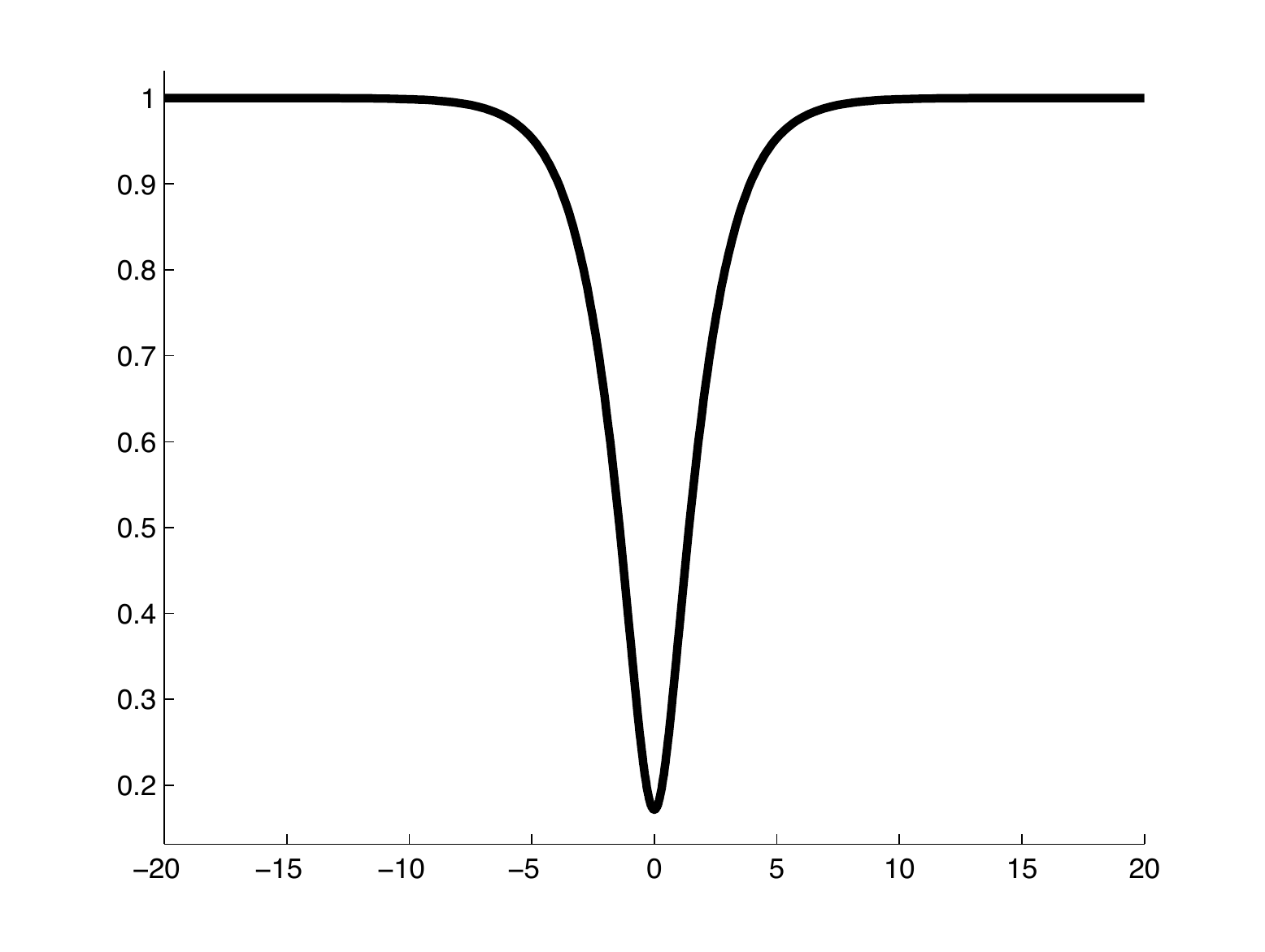} \quad (c) \includegraphics[scale=0.25]{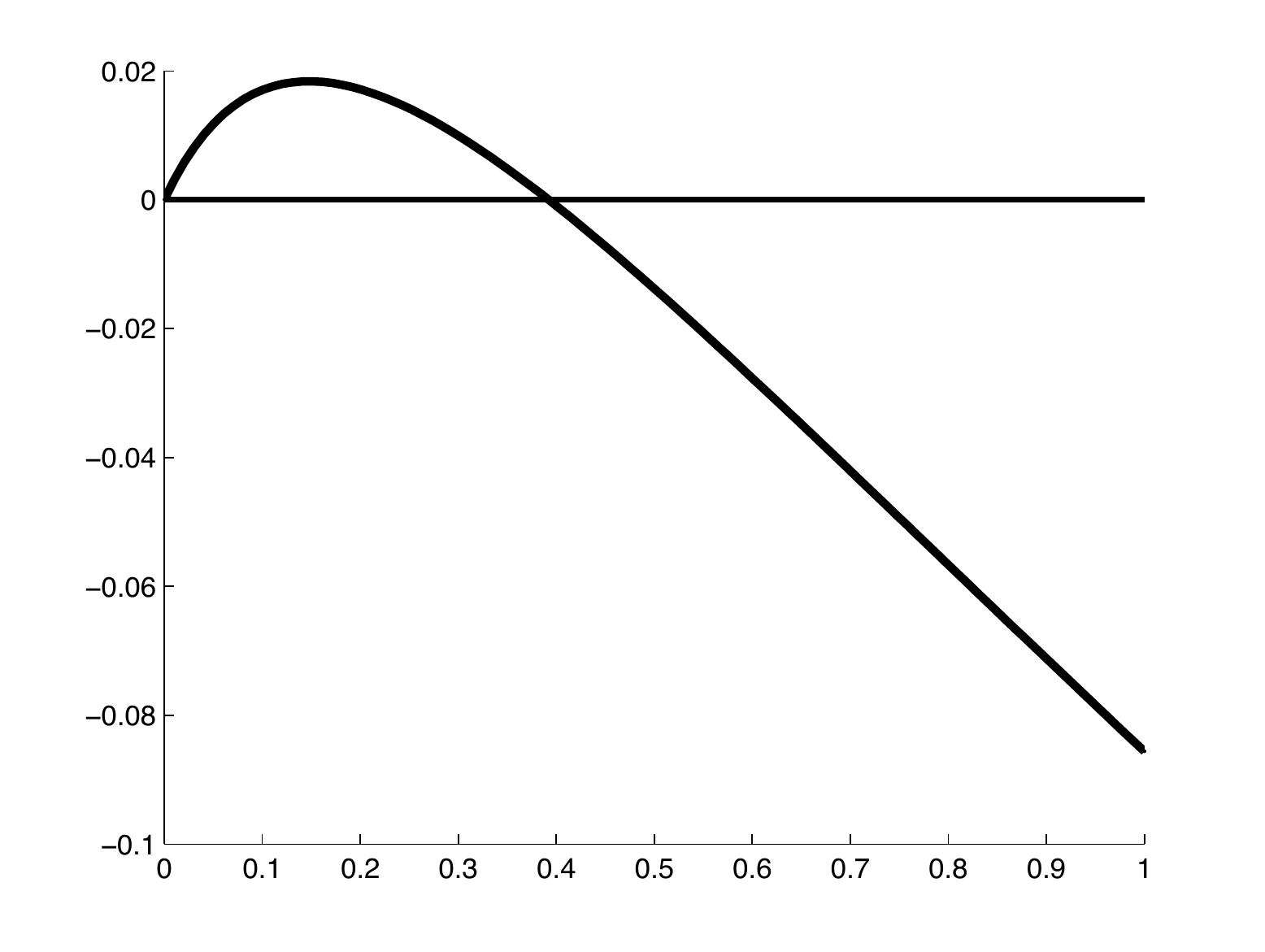}
 \end{array}
 $
\end{center}
  \caption{Evans function output for the Jin-Xin equations with $f(\tau)=-\tau^{-1/2}$ and $c_s=1$, $q=2$, and $H=-1/2$. We display the homoclinic profile
  in phase space in Figure (b) and against the spatial variable in Figure (b).   In Figure (c),
  we observe the Evans function output evaluated on the real line, illustrating the existence of an unstable
  positive eigenvalue.   
}
  \label{JinXin_Evans_2}
\end{figure}

An alternate way of numerically seeing this instability through the use of an Evans function is through the use of a
winding number calculation.  Indeed, the analyticity of the Evans function on the spectral
parameter $\lambda$ implies the number of solutions of the equation $D(\lambda)=0$ within the bounded component
of a contour $\Gamma$, along which $D|_{\Gamma}$ is non-vanishing, can be calculated via the winding number
\begin{equation}\label{winding}
\frac{1}{2\pi i}\int_{\Gamma}\frac{D'(\lambda)}{D(\lambda)}d\lambda.
\end{equation}
Notice that unlike our previous Evans function calculation, which detected only real roots of $D(\lambda)=0$,
this method allows one to exclude the possibility of unstable eigenvalues within a given compact domain.  This advantage
will be exploited below in conjunction with the eigenvalue bounds of Corollary \ref{HFcorollary} to conclude not only
the absence of real unstable eigenvalues but of any unstable eigenvalues.  To illustrate the numerical procedure, we consider
again the Jin--Xin profile studied above.  Notice that finding the analytically varying initializing bases $\{R_j^{\pm}\}$ as
described above preserves analyticity of the Evans function allowing us to use, via the
argument principle, a winding number computation to determine the presence of roots inside a given contour.  For this example,
we choose a semi-circular contour of radius one and inner circle radius $10^{-3}$.  Notice for a winding number calculation
the contour must not go through the origin since $D(0)=0$,
and hence we must deform the standard semi-circular contour slightly near the origin.  Using sufficient number of mesh points then,
we can ensure the relative change in $D(\lambda_j)$ between consecutive mesh points $\{\lambda_j\}$ is less than $0.2$, thus
assuring an accurate winding number count by Rouch\'e's Theorem.  For this specific example, a plot of the Evans function output
restricted to the semi-circular contour described above is shown in Figure \ref{JinXin_Evans_3}.  From this, it is easily seen that the winding number
is one, again verifying the expected instability.
\begin{figure}[htbp]
\begin{center}
$
\begin{array}{lr}
    (a)\includegraphics[scale=0.25]{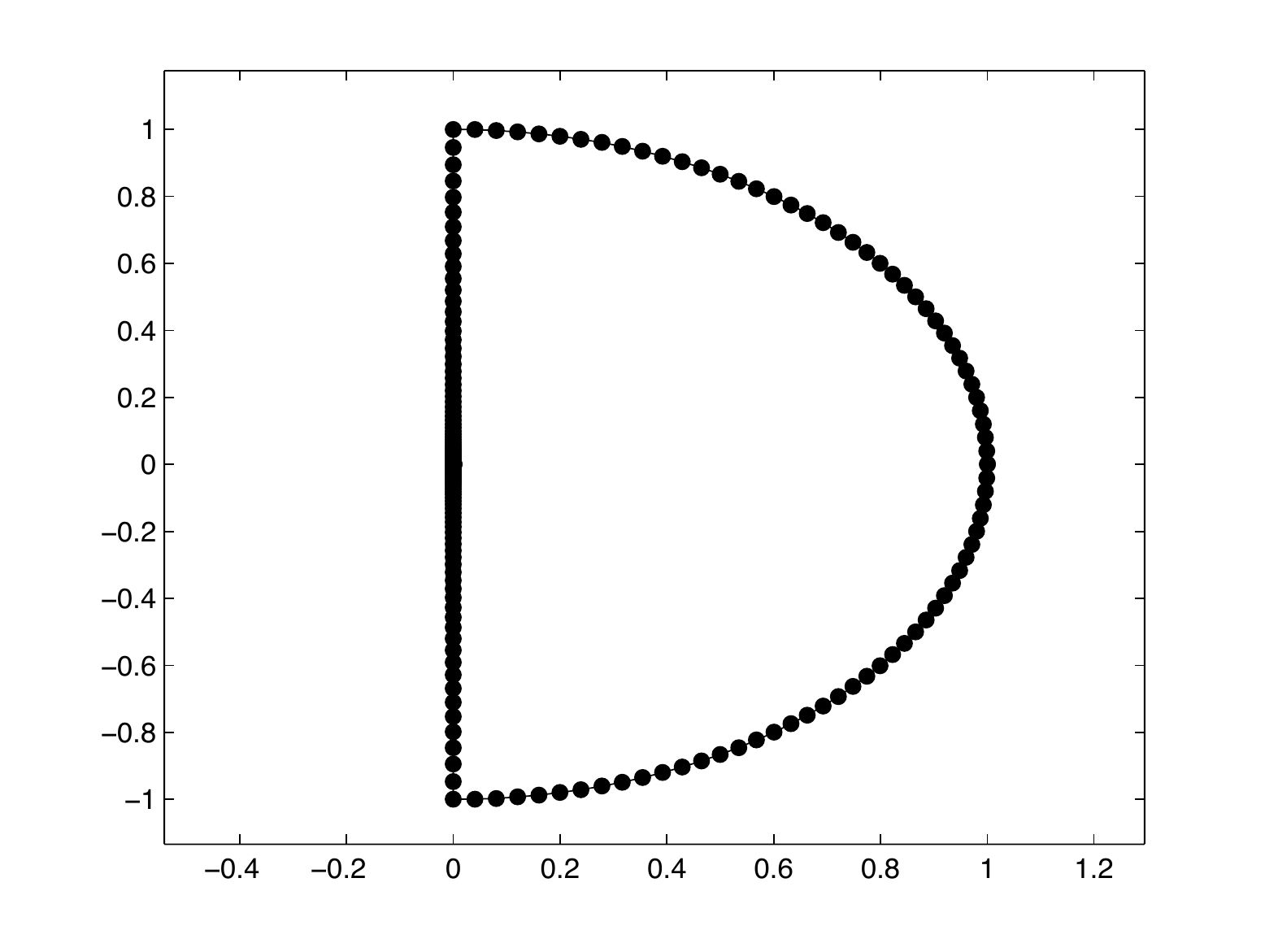}\quad (b)\includegraphics[scale=.25]{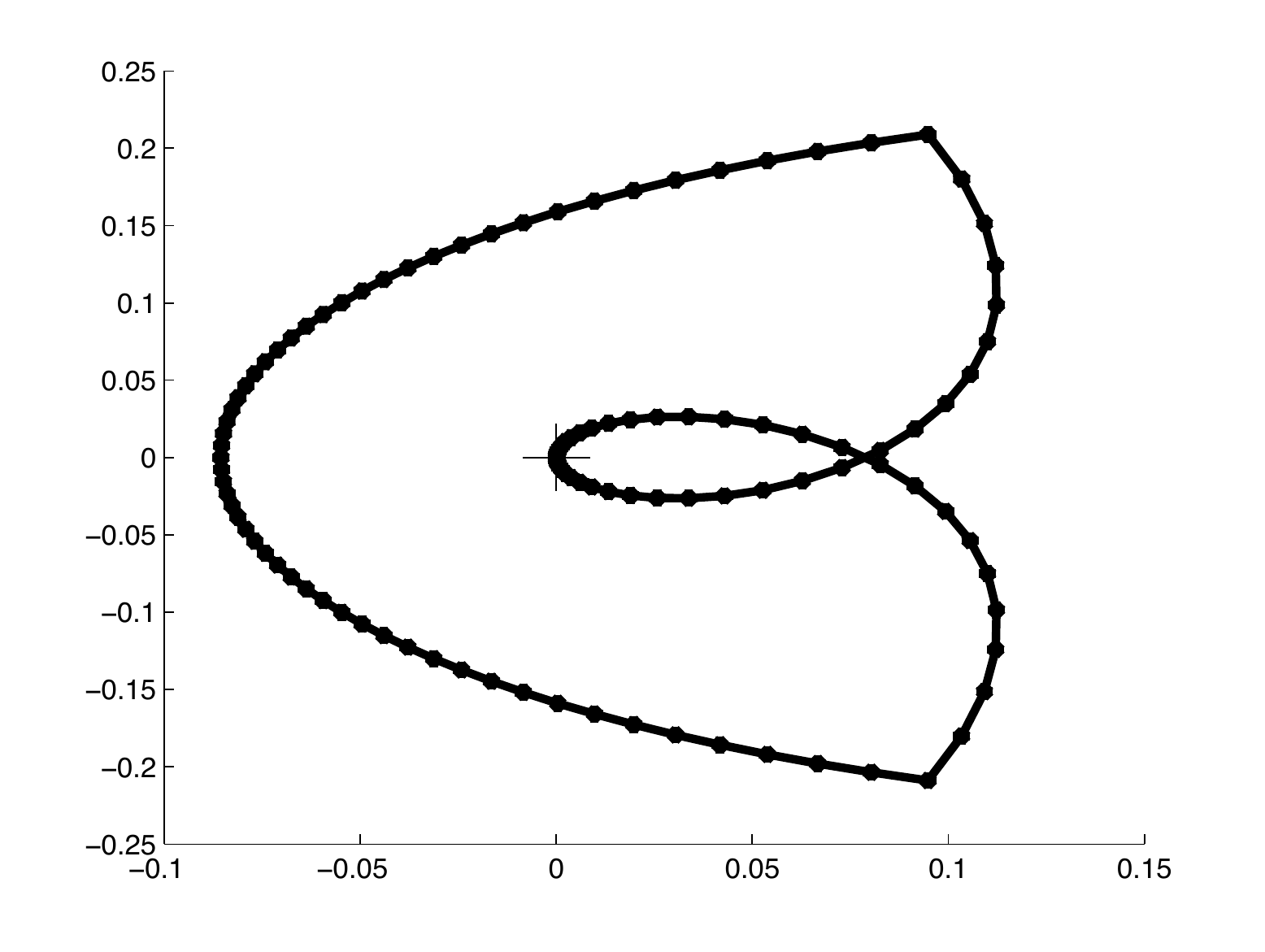} \quad(c) \includegraphics[scale=0.25]{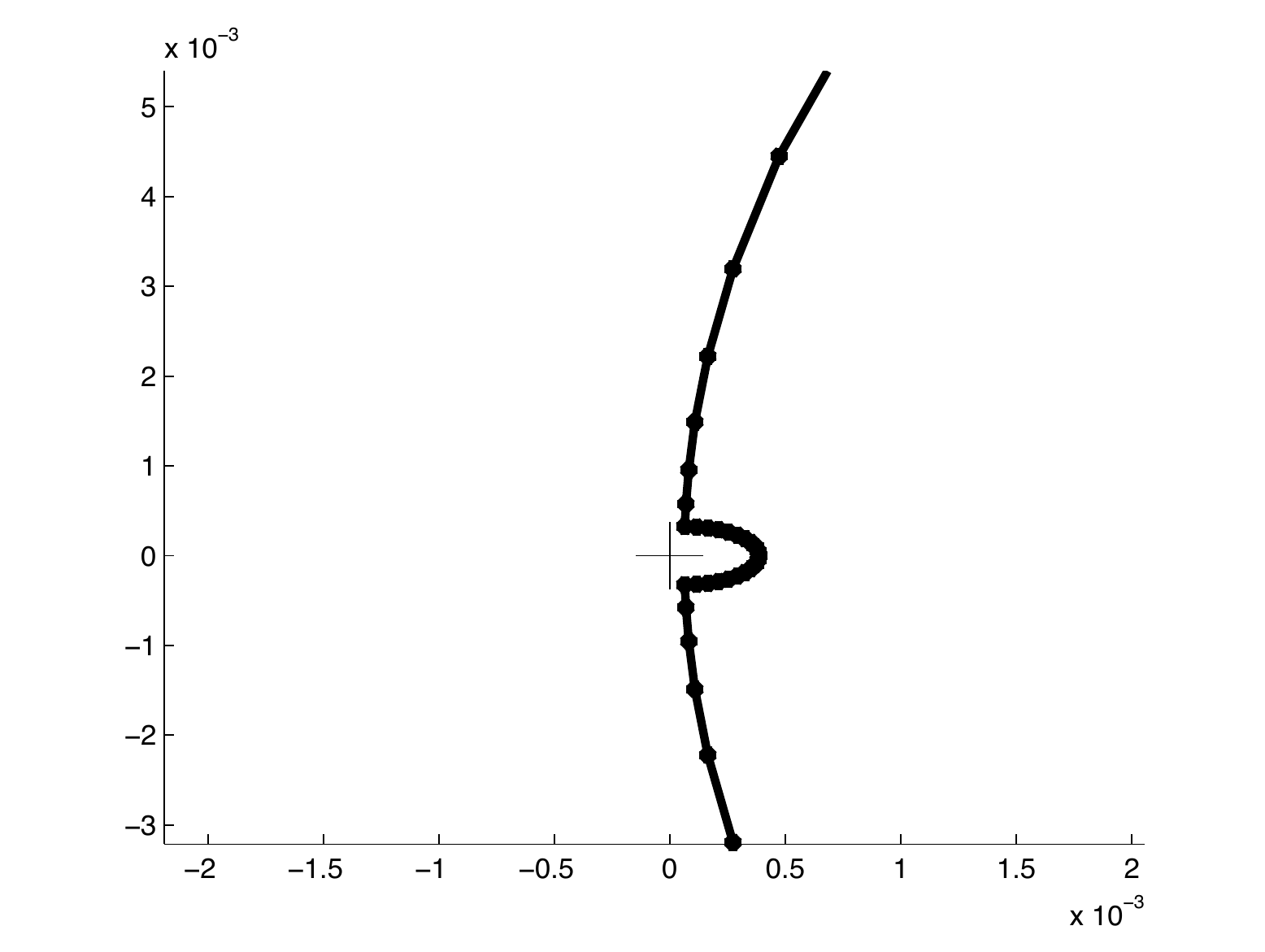}
 \end{array}
 $
\end{center}
  \caption{Evans function output for the Jin-Xin equations with $f(\tau)=-\tau^{-1/2}$ and $c_s=1$, $q=2$, and $H=-1/2$.
  Figure (a) depicts our semi-circular contour used for the winding number calculation.  In Figure (b), we have the Evans function output
  evaluated 
  along the contour and in Figure (c) we zoom in on the origin.  
  The winding number is one for the contour displayed, as expected.
}
  \label{JinXin_Evans_3}
\end{figure}

Next, we perform the analogous numerical study for the St. Venant equations \eqref{eqn:1conslaw} for two distinct values
of the turbulent friction parameters.  First, we consider the case $(r,s)=(1,1)$ demonstrating stable essential spectrum
and unstable point spectrum.  Second, we consider the common case $(r,s)=(2,0)$ demonstrate unstable essential spectrum
and stable point spectrum.  As mentioned in the introduction, this instability is convective in nature and we demonstrate
this with a time evolution study.

\subsection{Case $(r,s)=(1,1)$}

First, consider the equation
\ba \label{e:case1}
\tau_t - u_x&= 0,\\
u_t+ ((0.444)^{-1}\tau^{-2})_x&=
1- \tau^{2} u +20 (\tau^{-2}u_x)_x ,
\ea
corresponding to the St. Venant system \eqref{eqn:1conslaw} with $(r,s)=(1,1)$, $F=0.222$, and $\nu=20$.
In this example, we consider the traveling wave profile associated with the corresponding
traveling wave ODE \eqref{e:profile} with wave speed $c\approx3.1869$ and integration constant $q=1+c$.
Following the protocol previously described, in Figure \ref{venant_evan_homoclinic2}(a) we depict the corresponding
orbit in phase space along with the Evans function output along the real line.  From Figure \ref{venant_evan_homoclinic2}(b),
it is clear there exists an unstable real eigenvalue associated with this profile.  Although not depicted here,
this was further verified by a winding number calculation as in the Jin--Xin case above, in which
one finds a winding number of one for an appropriate semicircular contour.  Moreover, notice
that the limiting constant state $(1,0)$ satisfies the subcharacteristic condition
\[
2\sqrt{F}<\left(\lim_{x\to\pm\infty}\tau(x)\right)^{3/2}
\]
corresponding to spectral stability of the limiting state.  It follows that although this profile admits unstable point
spectrum, its essential spectrum is stable.  We verify this numerically using the SpectrUW package developed
at the University of Washington\cite{DK}, which is designed to find the essential spectrum of linear operators with
periodic coefficients by using Fourier-Bloch decompositions and Galerkin truncation; see \cite{CuD,CDKK,DK} for further information and for details concerning
convergence.  Although this package is designed for periodic problems, the essential spectrum corresponding
to the limiting constant state (treated as a periodic function) of our homoclinic is precisely the essential spectrum of the original homoclinic profile,
and hence we can numerically compute the essential spectrum of a solitary wave using SpectrUW in this way.
In particular, we find that the essential spectrum is stable, as expected; see Figure \ref{venant_evan_homoclinic2}(c).

\begin{figure}[htbp]
\begin{center}
$
\begin{array}{lr}
  (a)  \includegraphics[scale=.23]{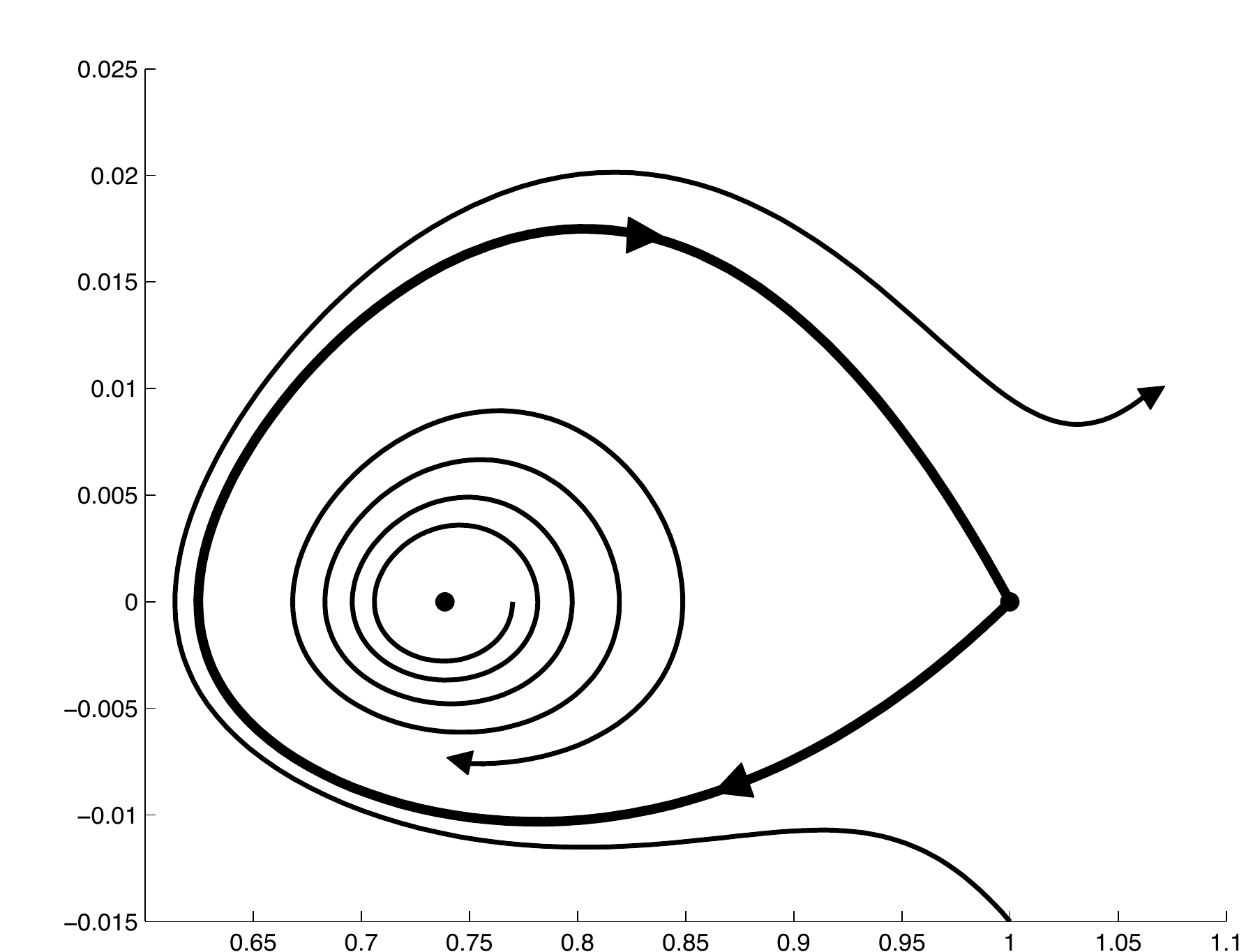}  
  \quad (b) \includegraphics[scale=0.23]{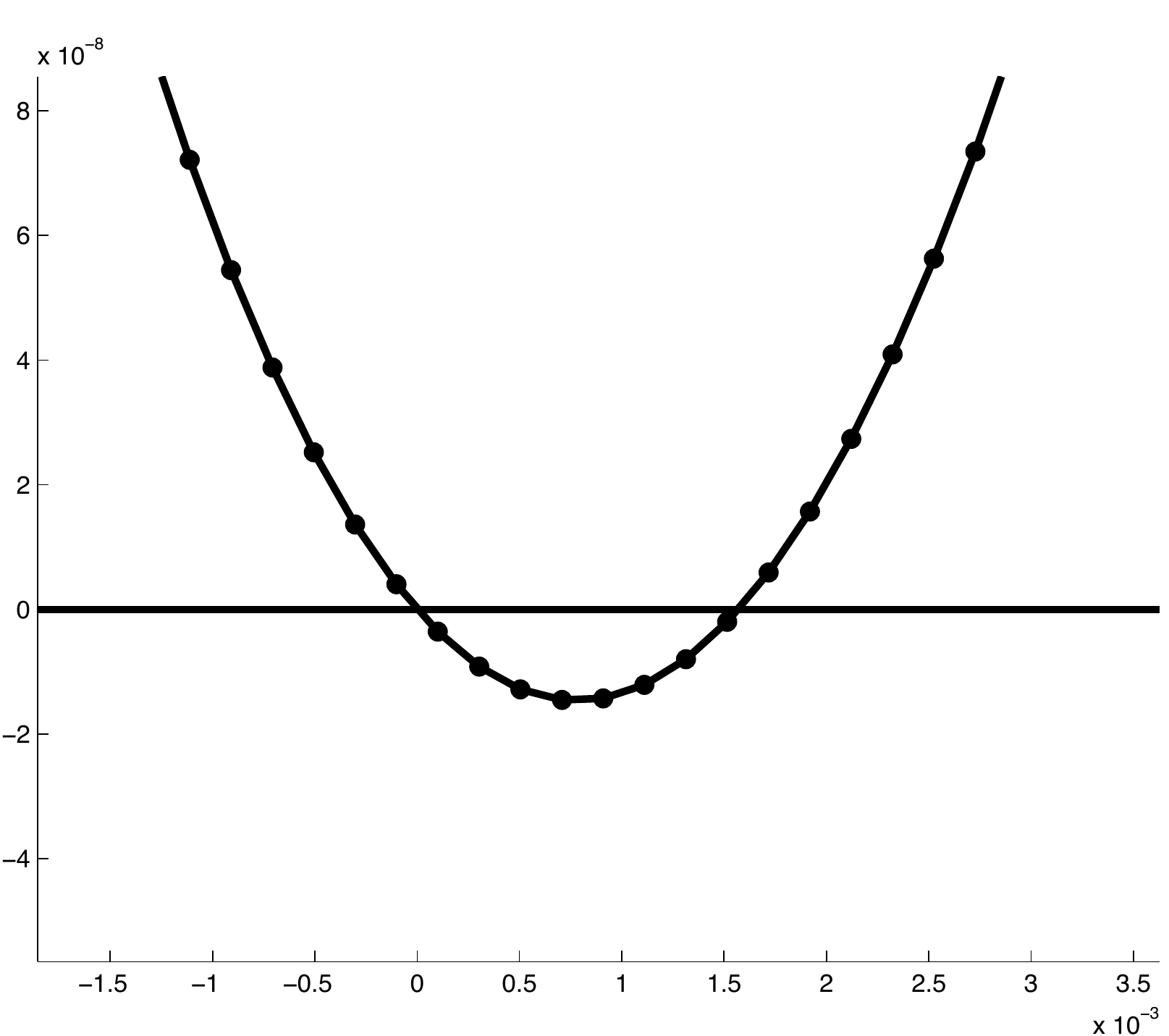}\quad (c)\includegraphics[scale=0.3]{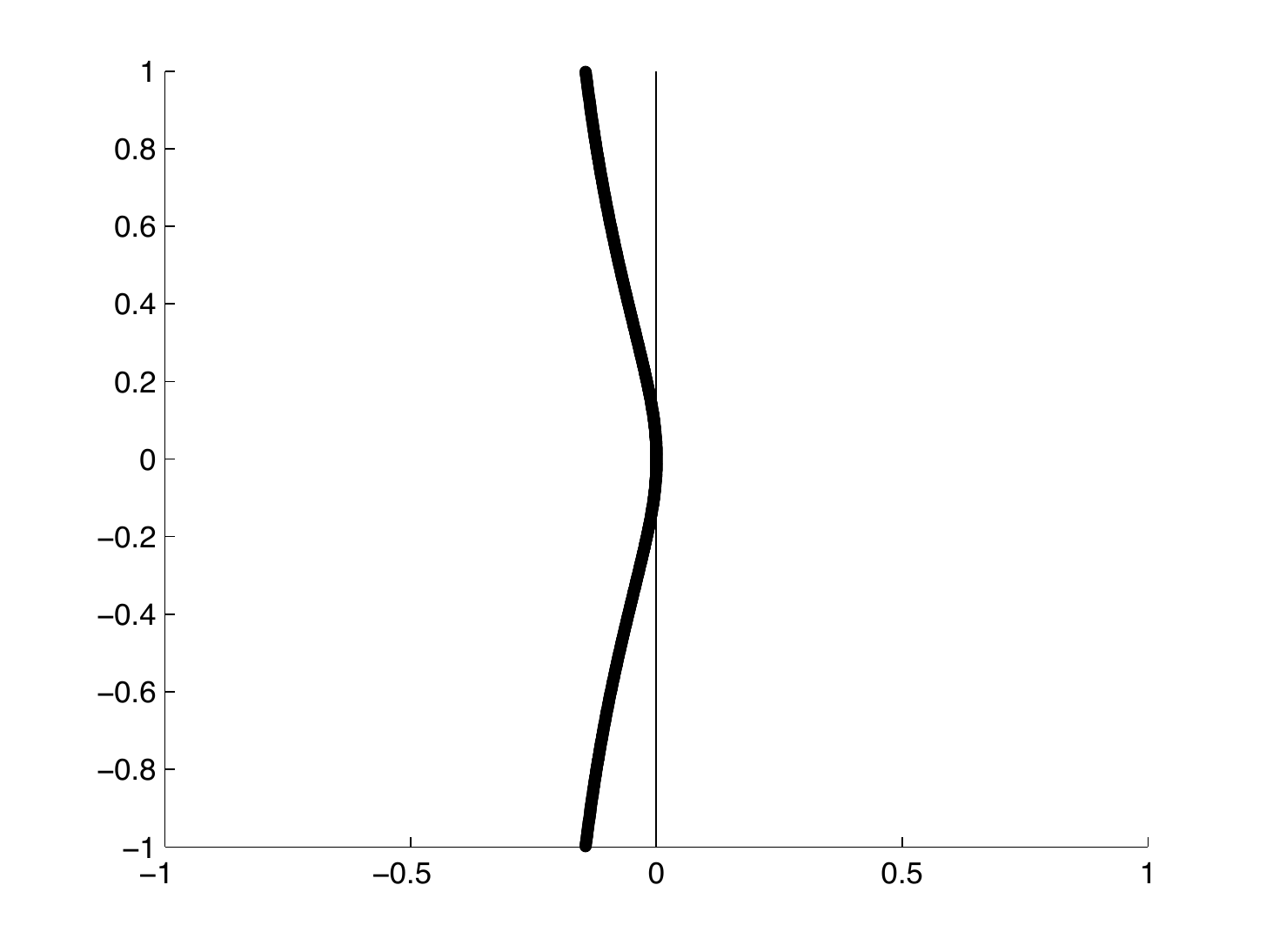}
 \end{array}
 $
\end{center}
  \caption{Demonstration of unstable point spectrum in generalized St. Venant system. Here $r=s=1$, $\nu=20$, $F=0.222$, $c\approx 3.1869$, and $q=1+c$.
  We have the phase portrait plotting $\tau$ verse $\tau'$ in Figure (a), 
  and the Evans function evaluated along the real line in Figures (b).  Finally, Figure (c) depicts the numerical evaluation of the essential
  spectrum of the homoclinic using the SpectrUW package.}
   \label{venant_evan_homoclinic2}
\end{figure}

We now wish to connect this numerical experiment to the analytical Evans function calculations of Section \ref{s:stabindex}, in particular
Corollary \ref{stabcor} and Remark \ref{odestab}.  To this end, we note that numerically it is found that periodics exist
for the system \eqref{e:case1} as $c$ is increased through the homoclinic wave speed, which here we have chosen as $c_{\rm hom}\approx 3.1869$.
Furthermore, as seen by Figure \ref{venant_evan_homoclinic2}(a) the equilibrium solution enclosed by the homoclinic orbit is a repeller.
Together these considerations suggest that $\partial_c d(c,q)>0$ by our discussion in Remark \ref{odestab}.
Note this is verified by noting from Figure \ref{venant_evan_homoclinic2}(b) that $D'(0)<0$ and using Lemma \ref{l:lowfreq}.
Thus, the instability of this profile follows trivially from direct analysis of the corresponding phase portrait and Corollary \ref{stabcor}.
Notice, however, that this example shows the dependence on the ``repeller/attractor" rule of thumb on how the periodic
orbits are generated from the homoclinic.

\subsection{Case $(r,s)=(2,0)$}\label{Case2}

Next, we consider the equation
\ba \label{e:case2}
\tau_t - u_x&= 0,\\
u_t+ ((18)^{-1}\tau^{-2})_x&=
1- \tau u^2 +0.1 (\tau^{-2}u_x)_x ,
\ea
corresponding to the St. Venant system with $(r,s)=(2,0)$, $F=9$, and $\nu=0.1$.  Here, we consider the solitary wave profile
associated with the corresponding profile ODE with wave speed $c\approx 0.7849$ and integration constant $q=1+c$,
which is depicted in phase space in Figure \ref{venant_evan_homoclinic}(a).  We begin by discussing the point spectrum
associated to this solitary wave solution.  Noting that periodics are seen to be numerically generated as the wave speed
is decreased from the homoclinic speed, the fact that the equilibrium enclosed in the homoclinic is a repeller suggests
that $\partial_c d(c,q)<0$ by Remark \ref{odestab}.  In particular, our heuristic suggests that the stability
index derived in Section \ref{s:stabindex} does not yield any information concerning the existence of unstable
point spectrum of the corresponding linearized operator.  This motivation is further supported noting
from the Evans function output along the real line in Figure \ref{venant_evan_homoclinic}(b) that $D'(0)>0$, and hence
$\sgn\partial_cd(c,q)<0$ by Lemma \ref{l:lowfreq} as expected.

In order to further study the point spectrum of the corresponding linearized operator,
we need a high frequency bound on the eigenvalues of \eqref{e:homlinII}.  By
using the numerically computed profile to evaluate the functions $a_j$ in \eqref{ab},
an application of Corollary \ref{HFcorollary} implies that
any unstable roots of $D(\lambda)$, if they exist for the given parameter values,
must satisfy the bound $|\lambda|<308$.  In Figure \ref{venant_evan_homoclinic}(c) and (d) the
output of the Evans function evaluated on a semi-circle of outer radius $308$ and inner radius $10^{-3}$,
from which we compute that the winding number \eqref{winding} is equal to zero\footnote{In this example, the Evans function was evaluated
on the contour with 3457 mesh points, yielding a maximum relative error between mesh points of 0.2.  As stated
before, this is sufficient to yield an accurate winding number by Rouch\'e's Theorem.}.  In particular,
it follows that the linearized eigenvalue problem in consideration has stable point spectrum as expected.  As noted
in the introduction, this example is notable as being the first example of a second-order hyperbolic-parabolic
conservation or balance law which admits a solitary wave with stable point spectrum.

\br
The eigenvalue estimates of Corollary \ref{HFcorollary} should be contrasted
with the results of a convergence study based on \eqref{conv}.
In the example considered here, the results of a high frequency convergence
study are listed below, with ``Relative Error'' denoting the maximum
relative error between the numerically computed value of $D(\lambda)$
and an approximant \eqref{conv} determined by curve fitting on the
semicircle $|\lambda|=R$, $\Re \lambda \ge 0$:
\[
\begin{array}{cc}
\textrm{Relative Error} & \textrm{Radius R}\\
  0.1112 & 8192 \\
  0.1525 & 4096 \\
  0.2062 & 2048 \\
  0.2740 & 1020 \\
  0.4230 & 308 \\
  0.4522 & 250 \\
  0.5484 & 130
\end{array}
\]
The radius $R=308$ derived from our high frequency asymptotics and tracking in Corollary \ref{HFcorollary} thus
correspond to a relative error of approximately $R=0.4230$, and hence the estimates in Corollary \ref{HFcorollary}
are quite efficient even though they produce a radius that is relatively large.  Furthermore, from this convergence study
it seems then that the convergence rate of $\mathcal{O}(|\lambda|^{-1/2})$ from \eqref{conv} breaks down for $R\approx 250$.
This suggests that the bound of $R=308$ obtained from our tracking estimates in Corollary \ref{HFcorollary}
is very close to the boundary of the high frequency asymptotics.
\er

Next, recalling that the subcharacteristic condition reduces to $F<4$ in this case we see that the
solitary wave must have unstable essential spectrum.  This is verified numerically in Figure \ref{venant_evan_homoclinic}(e),
where again we approximated the solitary wave with a periodically extended version of itself with very large
period and used package SpectrUW.  However, by performing a time evolution study of the generated
homoclinic profile we find that the nature of this instability is seen to be convective.  That is, a small perturbation
leaves the original profile relatively unchanged in shape, but grows
as an oscillatory time-exponentially growing Gaussian wave packet
after emerging from the profile; see Figure \ref{time_evol}.  For this numerical study, we used a Crank-Nicholson finite difference
scheme, i.e. we used a forward difference time derivative and a centered, averaged spatial derivative approximation.
This solitary wave is thus seen to be metastable in the sense that the perturbed solitary wave propagates relatively
unchanged down the ramp, while shedding oscillatory instabilities in its wake.  In the next section, we further
examine this phenomenon and its relation to the stability of nearby periodic waves.
\begin{figure}[htbp]
\begin{center}
$
\begin{array}{lr}
  (a)  \includegraphics[scale=.25]{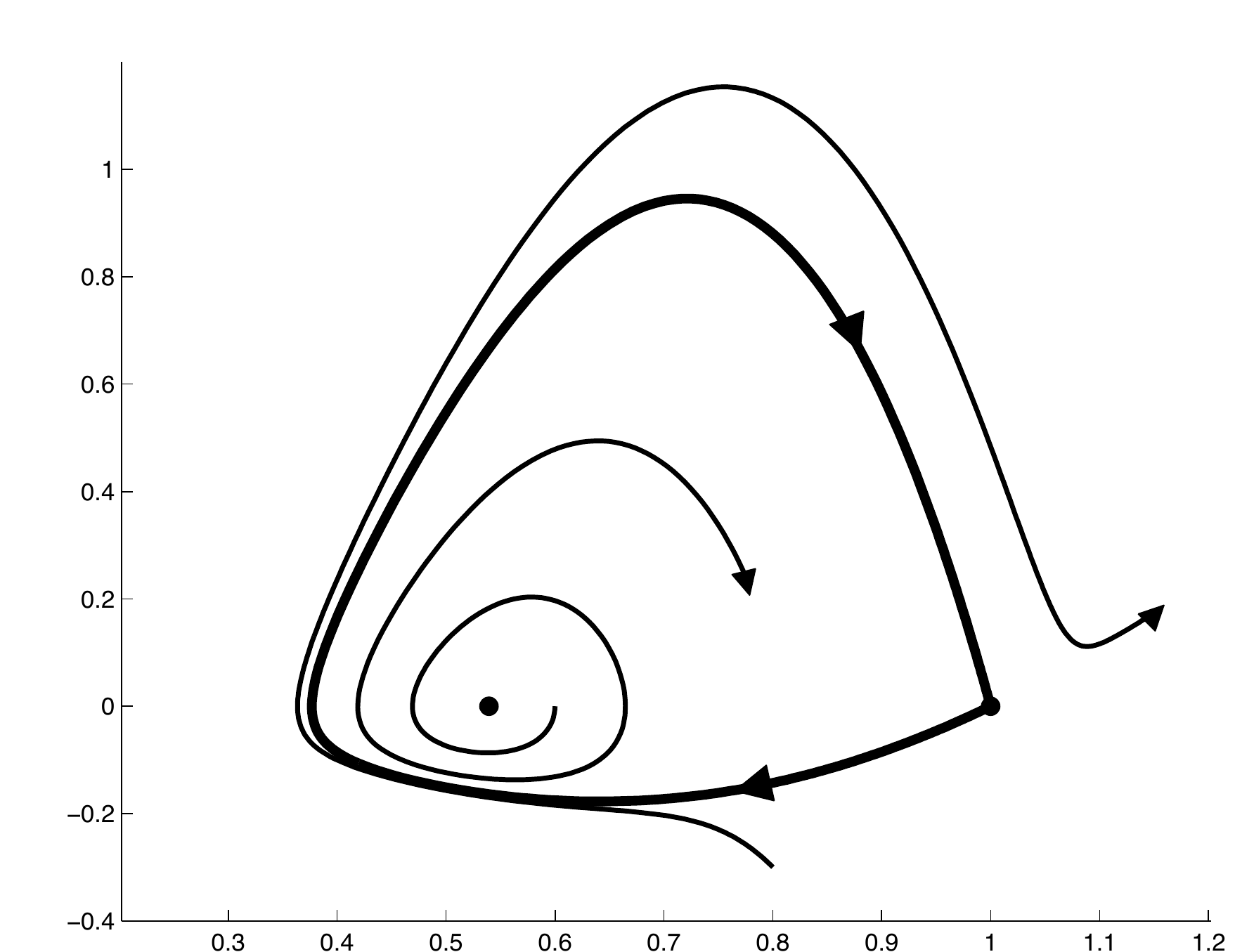} \quad 
  (b) \includegraphics[scale=0.25]{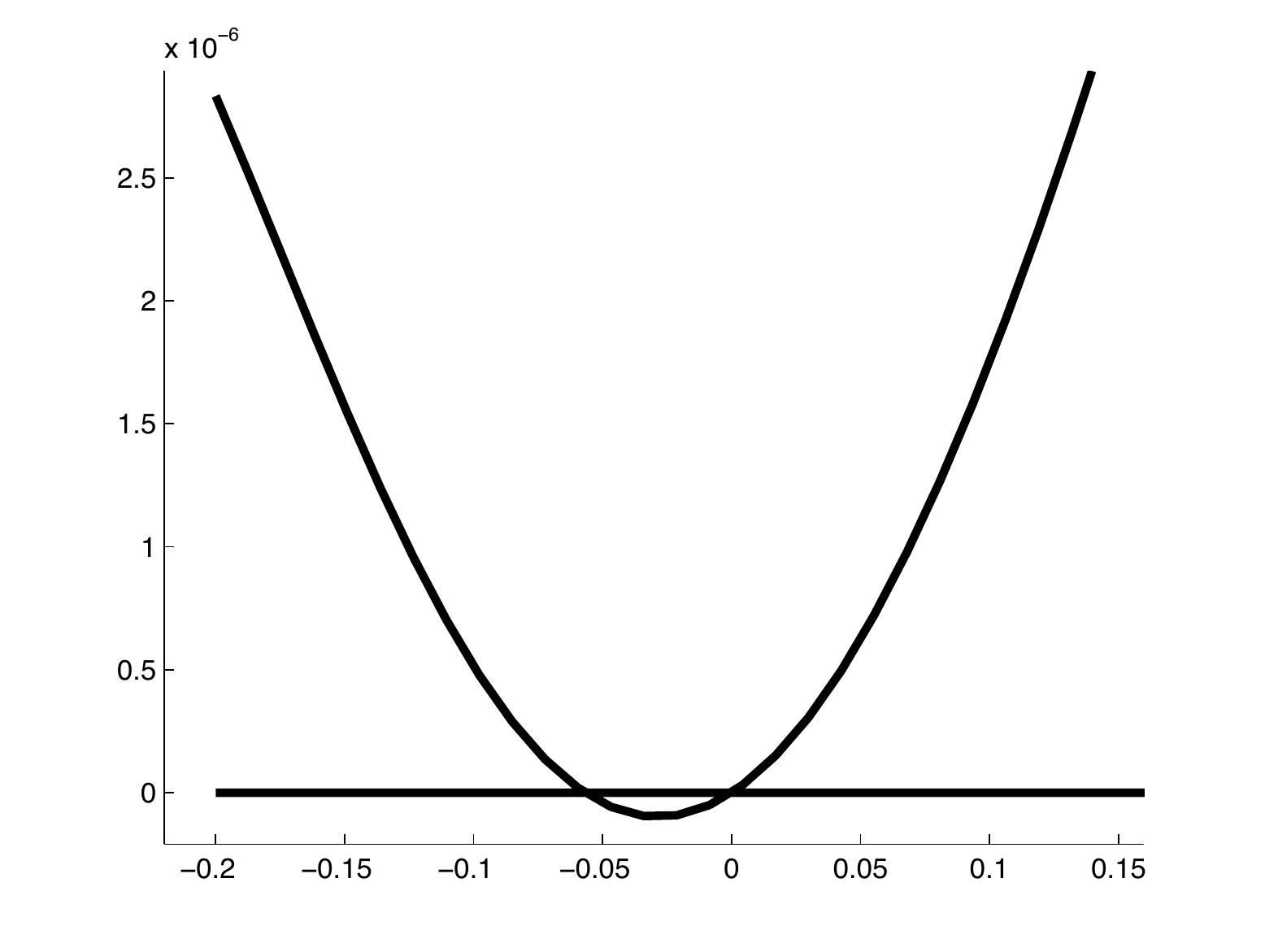}\quad (c)  \includegraphics[scale=.25]{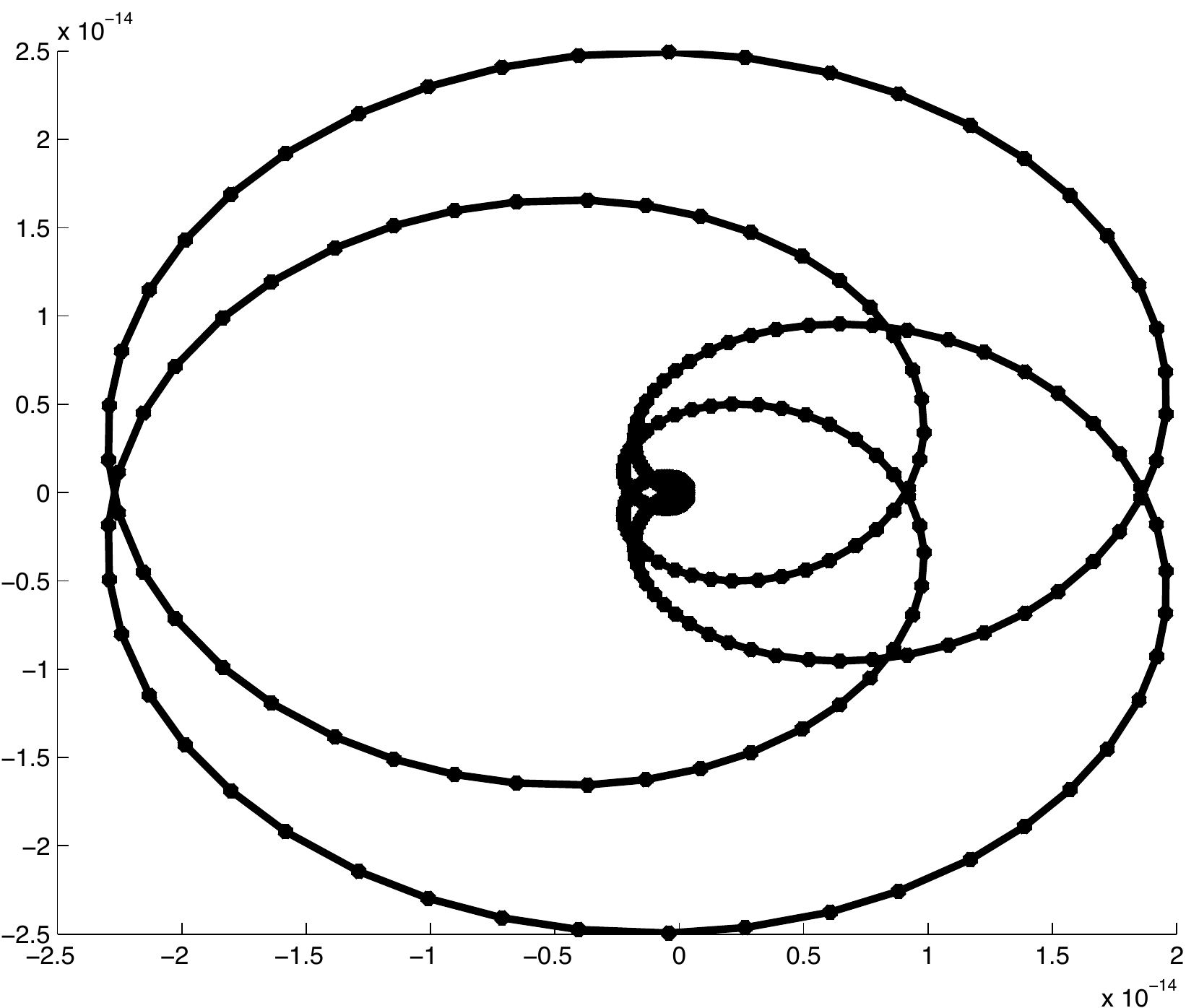}\\
  (d) \includegraphics[scale=0.25]{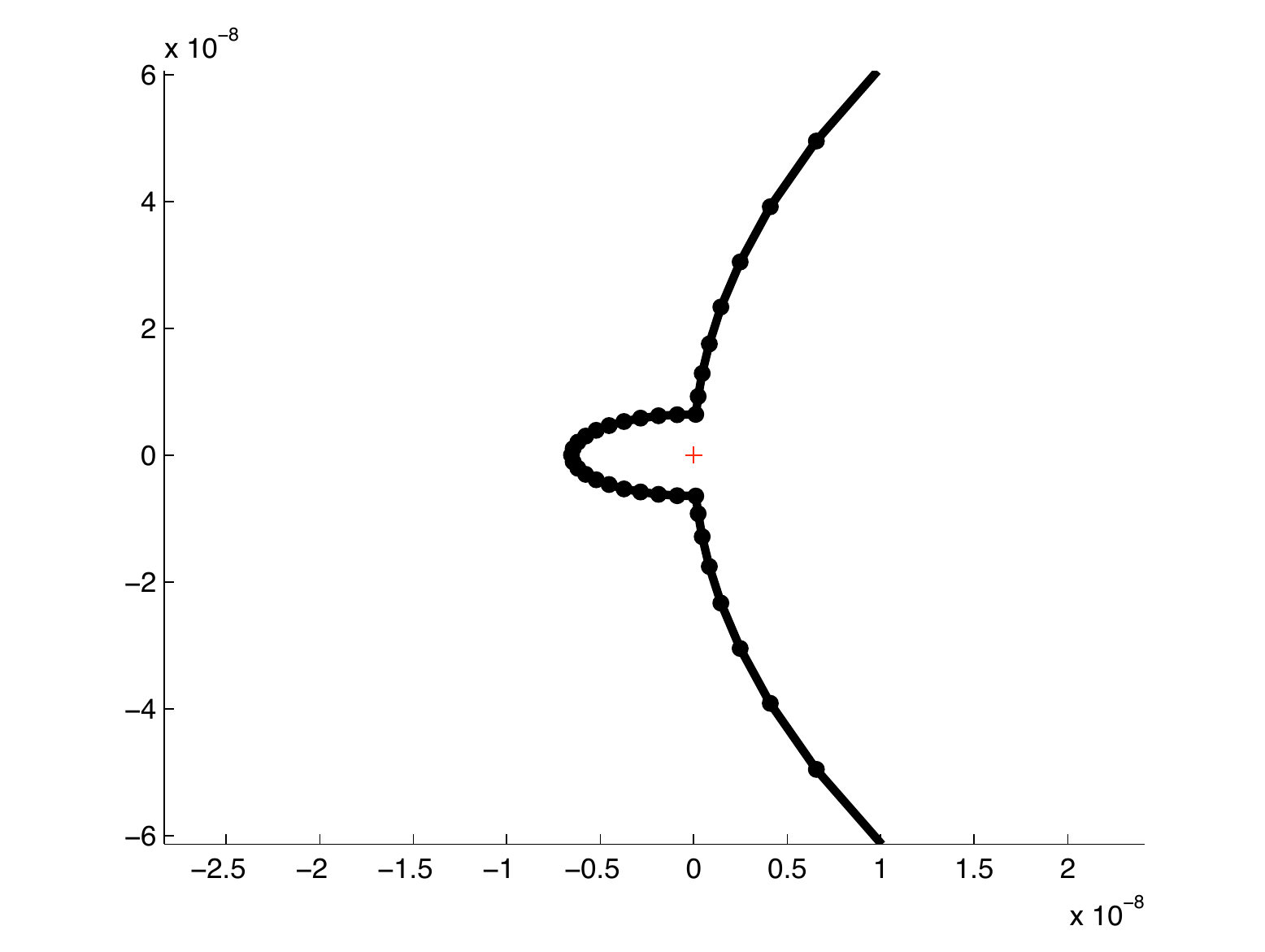}\quad (e)\includegraphics[scale=0.25]{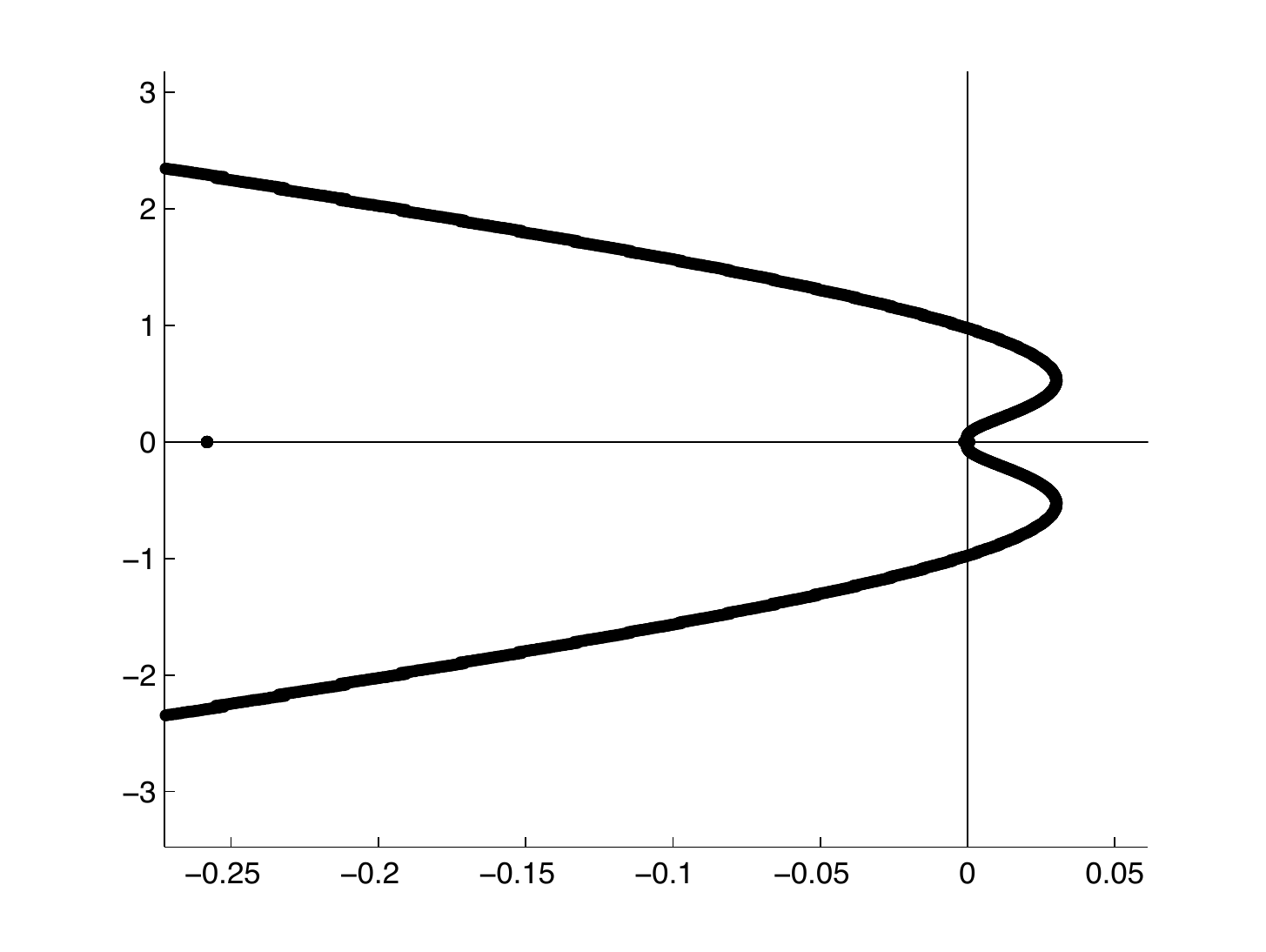}
 \end{array}
 $
\end{center}
  \caption{Evans function ouput for St. Venant system with $r=2$ and $s=0$ corresponding to a homoclinic profile. Here $c\approx 0.7849$, $q=1+c$, $F=9$, and
  $\nu=0.1$. In Figure (a) we have the phase portrait plotting $\tau$ verse $\tau'$.  In Figure (b) we plot the Evans function evaluated on the real line. 
  In Figures (c) we show the Evans function output evaluated on a semicircle of radius $R=308$ with a small inner circle of radius $10^{-3}$, where
  we zoom in on the origin in Figure (d). The winding number is zero in these computations.  Finally, Figure (e) depicts our approximation of the essential spectrum
  as generated by SepctrUW.}
\label{venant_evan_homoclinic}
\end{figure}

\begin{figure}[htbp]
\begin{center}
$
\begin{array}{lr}
  (a)  \includegraphics[scale=.25]{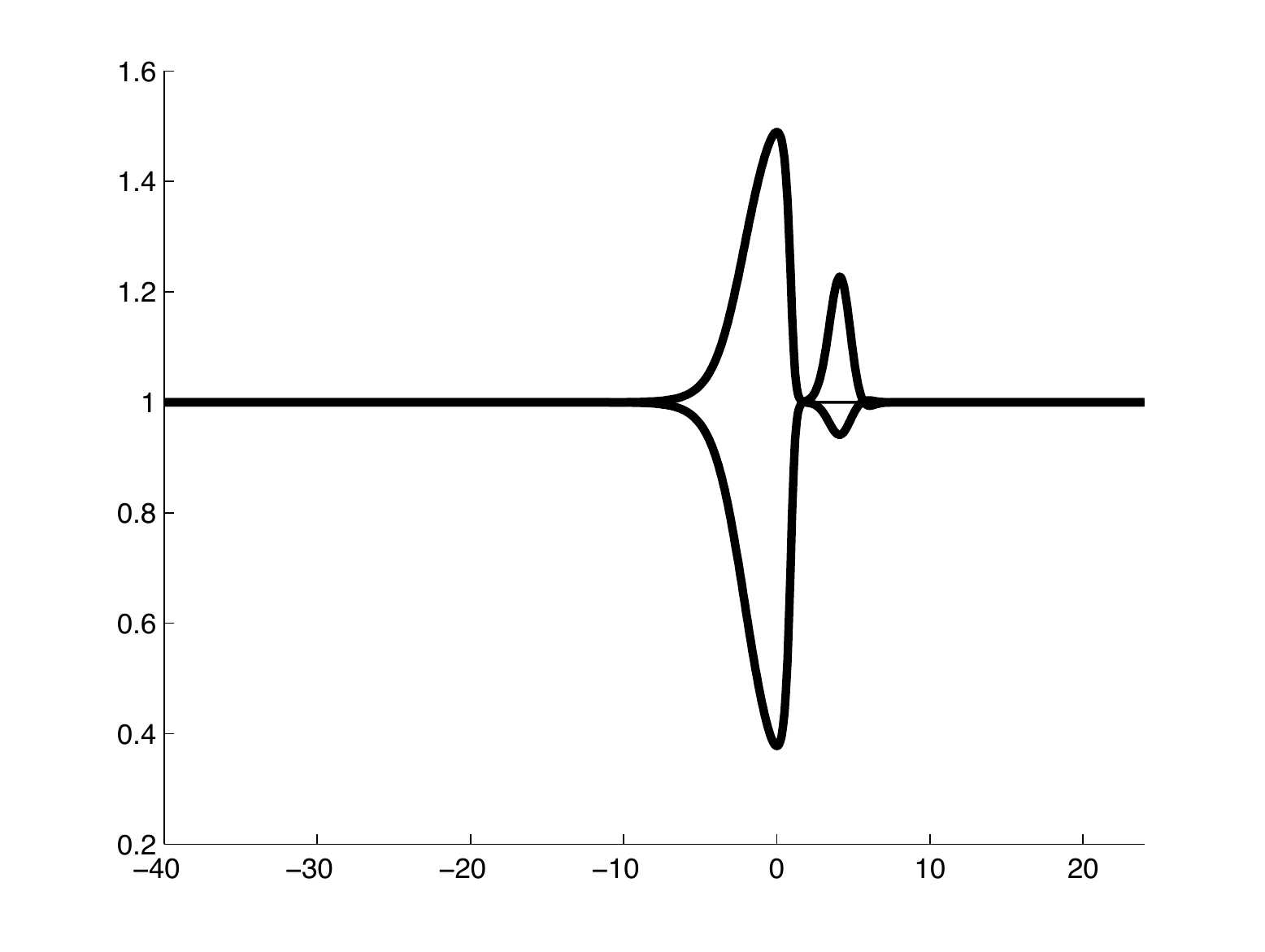}&(b) \includegraphics[scale=0.25]{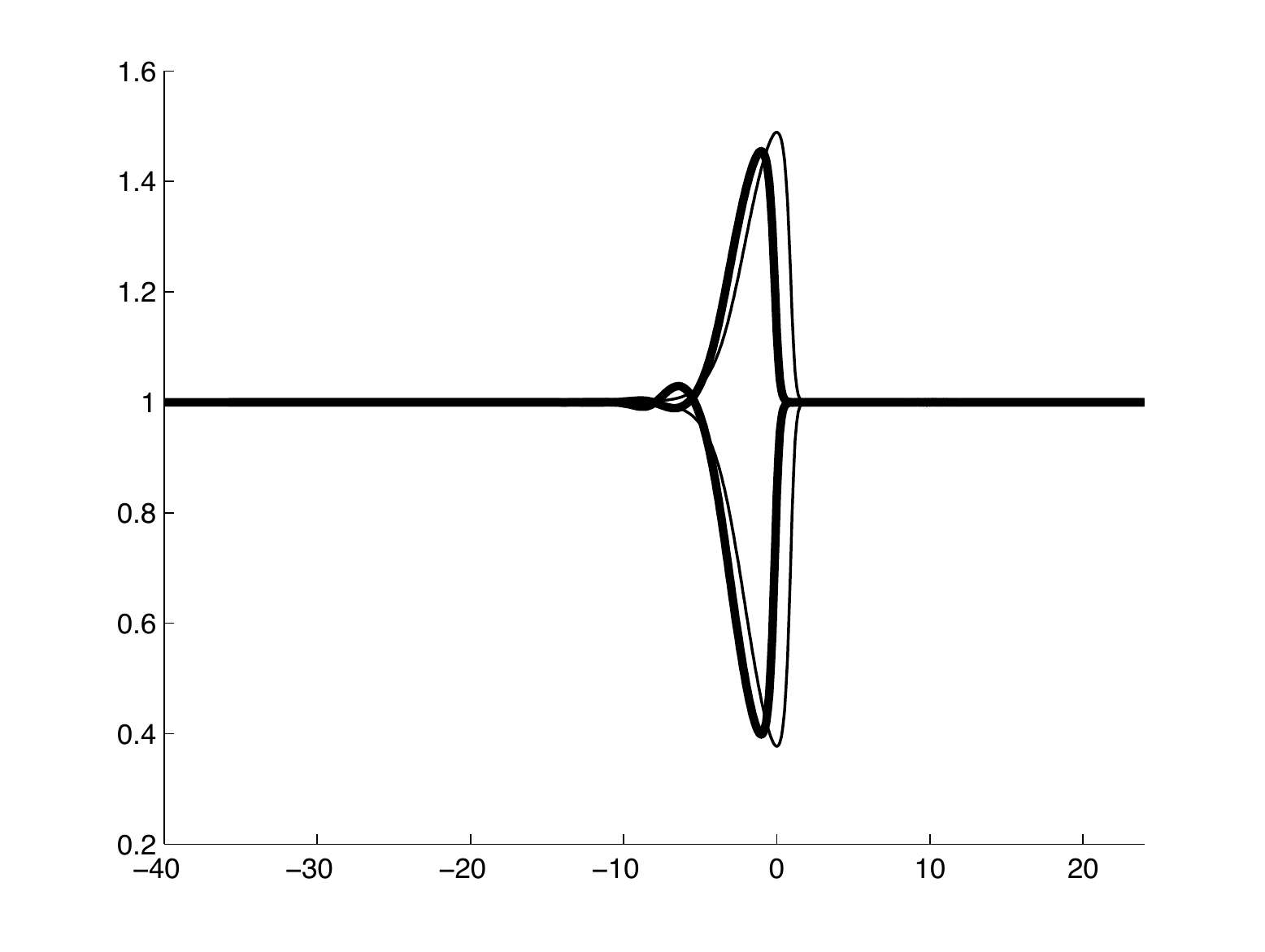}\\
  (c)  \includegraphics[scale=.25]{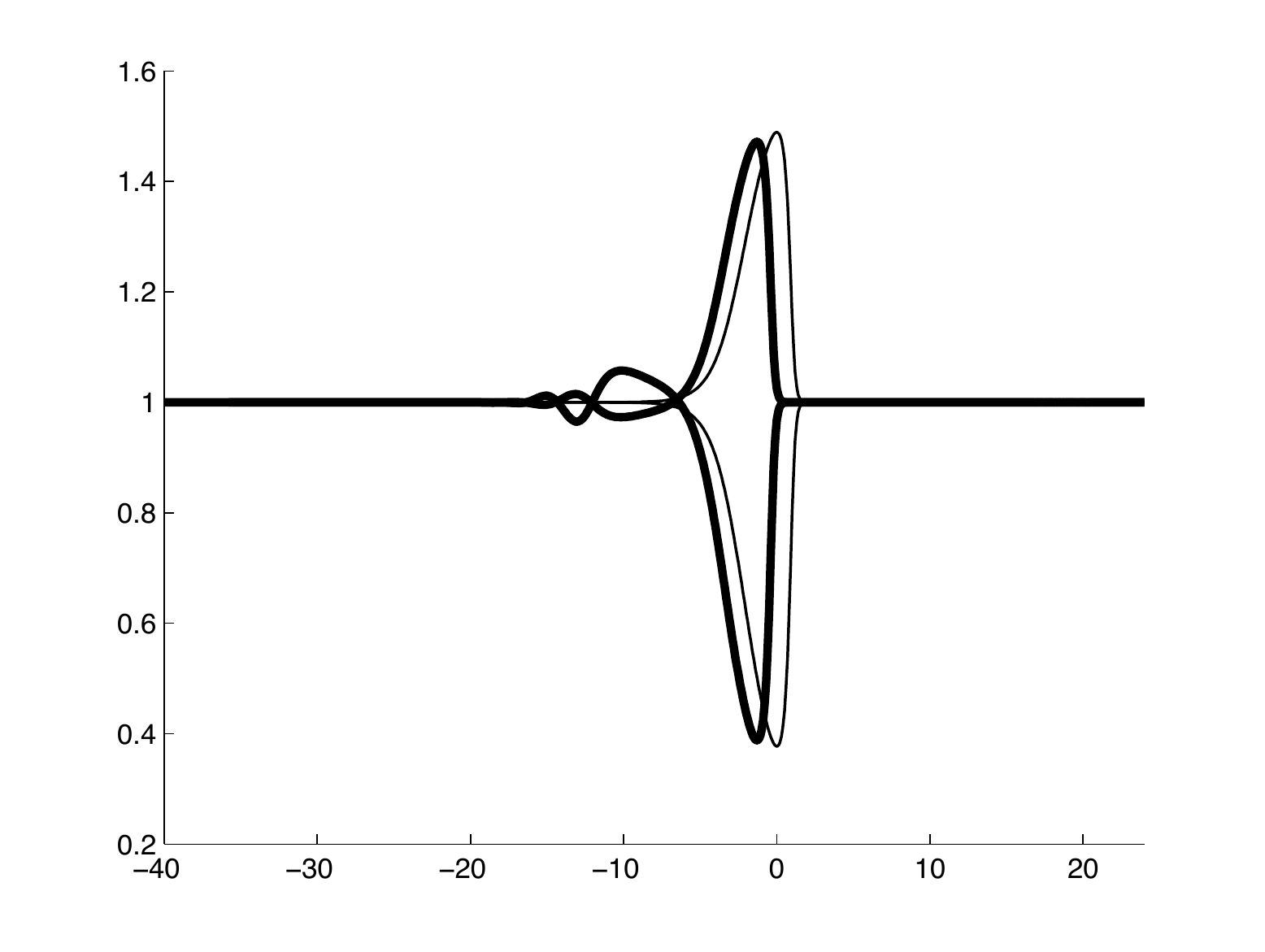}&(d) \includegraphics[scale=0.25]{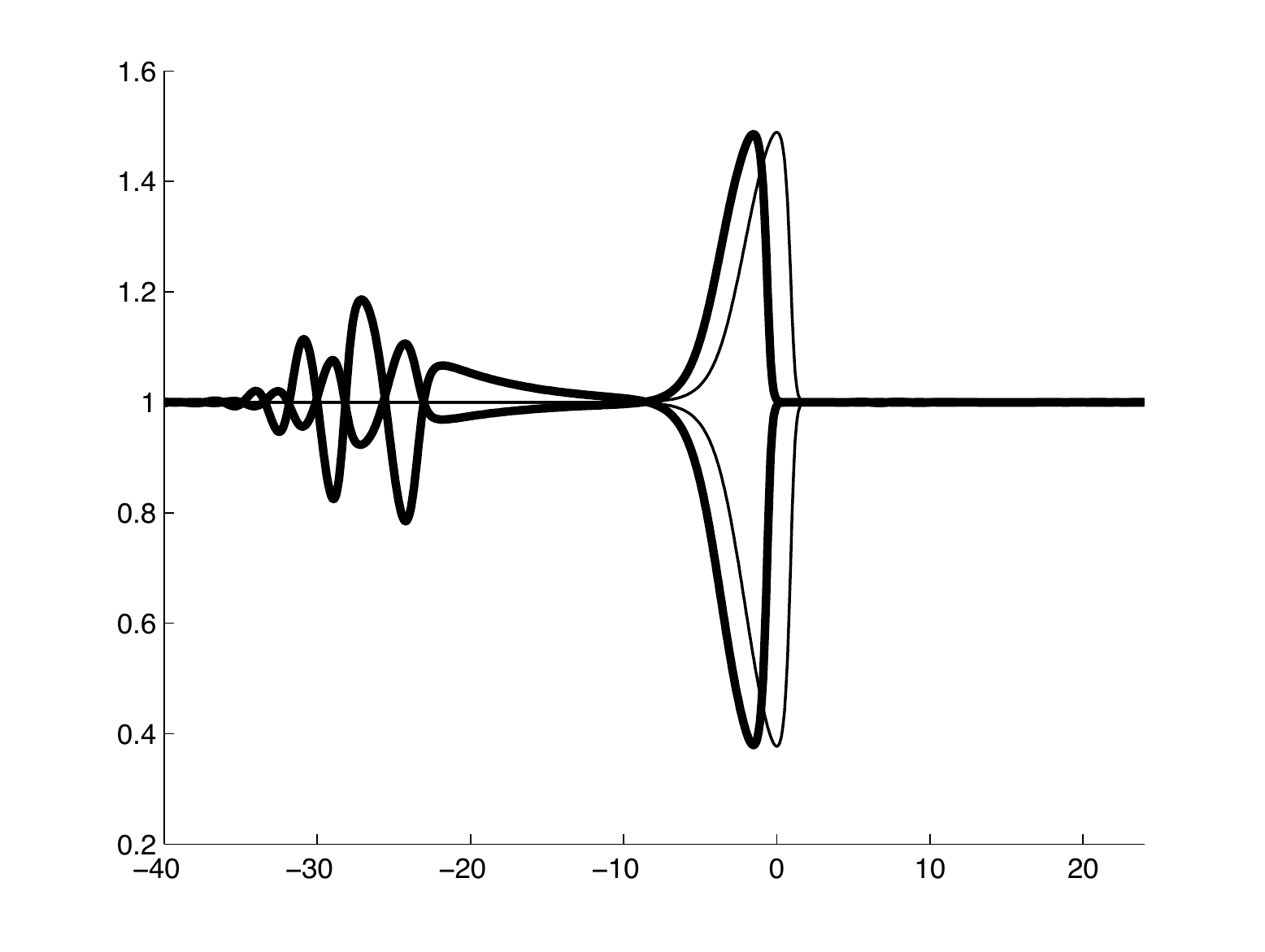}
 \end{array}
 $
\end{center}
  \caption{Time evolution
snap shots for St. Venant equations with $r=2$, $s=0$.
In Figure~(a) we have a perturbation to the right of the profile solution. In figure (b) we see the perturbation moving left emerge from the profile and grow as Gaussian curves displayed in Figures (c) and (d). The solution settles to a translate of the profiles (thin lines). Here $c\approx 0.7849$, $q=1+c$, $F=9$, and $\nu=0.1$.    }
 \label{time_evol}
\end{figure}

\section{Dynamical stability and stabilization of unstable waves}\label{s:dyn}

As noted in the introduction, similar metastable phenomena to that observed in Section \ref{Case2} have been observed
by Pego, Schneider, and Uecker \cite{PSU} for the related
fourth-order diffusive Kuramoto-Sivashinsky model
\be\label{KSeq}
u_t+\partial_x^4u+\partial_x^2u+\frac{\partial_x u^2}{2}=0,
\ee
an alternative model for thin film flow down a ramp.
They describe asymptotic behavior of solutions of this model as dominated
by trains of solitary pulses.  
%
Indeed, such wave trains are easily observed on any rainy day in runoff
down a rough asphalt gutter
as found by many roads in the U.S.
(e.g., near the last author's home), as are small oscillations
between pulses as might be suggested by the analyses here and in \cite{PSU},
corresponding to convective instabilities.

Thus, both of these sets of results
may be regarded as partial explanation of the somewhat surprising
phenomenon that asymptotic behavior of actual inclined thin-film flow
is dominated by trains of pulse solutions which are themselves unstable.
Regarding this larger issue, we have a
further observation that
we believe completes the explanation of this interesting puzzle,
at least at a level of heuristic understanding.

The key is to reformulate the question as:
how can we explain observed stable behavior of trains of solitary waves
that are in isolation exponentially unstable?
Or, more pointedly: {\it how can a train of solitary pulses
stabilize the convective instabilities shed from their neighbors?}
Phrased in this way, the question essentially answers itself: it must be
that the local dynamics of the waves
are such that convected perturbations
are {\it diminished} as they cross each solitary pulse,
counterbalancing the growth experienced as they traverse the
interval between pulses, on which they behave as perturbations
of an unstable constant solution.
%
This diminishing effect is clearly apparent visually in time-evolution studies;
see Figure \reff{time_evol}.
However, it is a challenge to quantify it mathematically.  In particular, it is
only partially but not wholly encoded by the point spectrum that we
usually think of as determining local dynamics of the wave \cite{He,GZ,ZH}.

Rather, the relevant entity
appears to be the {\it dynamic spectrum}, defined
as the spectrum of the periodic-coefficient linearized operator about
 the periodic wave obtained by pasting together copies of a suitably
truncated solitary pulse.
Here, the choice of truncation is not uniquely specified, but should
intuitively be at a point where the wave profile has ``almost converged''
to its limiting endstate.
This dynamic spectrum would govern the behavior of an arbitrarily closely
spaced array of solitary pulses, so captures the diminishing property
if there is one.
A bit of thought reveals the difficulty of trying to capture diminishing
instead by decrease in some specified norm.  For, the decay we are expecting is
the diffusive decay of a solution of a heat equation, which does not occur
at any specified rate when considered from a given norm to itself, but
shows up in long-time averaged behavior from a more localized
(e.g., $L^1$) norm to a less localized (e.g., $L^2$ or $L^\infty$) norm,
and whose progress is difficult to measure by a ``snapshot'' at
an intermediate stage.

In Figure \reff{dynamic_spectrum} we depict the dynamic spectrum of the
profile used in the time evolution study of Section \ref{Case2}, along with
the periodically extended version of this profile used to carry out the
necessary numerics.
We see that the dynamic spectrum indeed appears to be quite stable,
despite the instability of the essential spectrum of the linearized
operator about the wave (corresponding to the convective instabilities
discussed above), supporting our picture of pulse profiles as
``de-amplifiers'' that can mutually stabilize each other when placed
in a sufficiently closely spaced wave train.

\begin{figure}[htbp]
\begin{center}
$
\begin{array}{lr}
  (a) \includegraphics[scale=0.35]{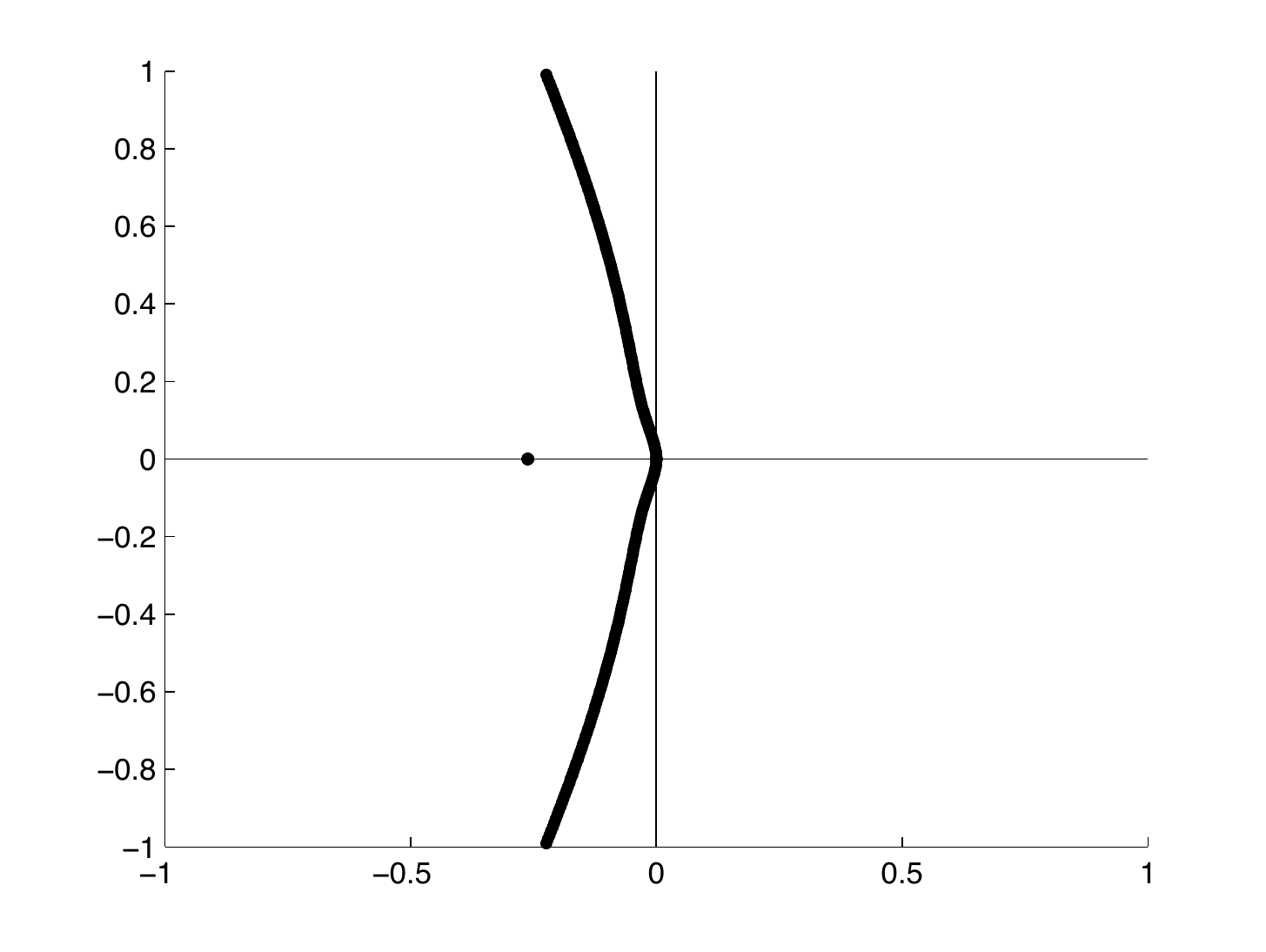} \quad (b)\includegraphics[scale=0.25]{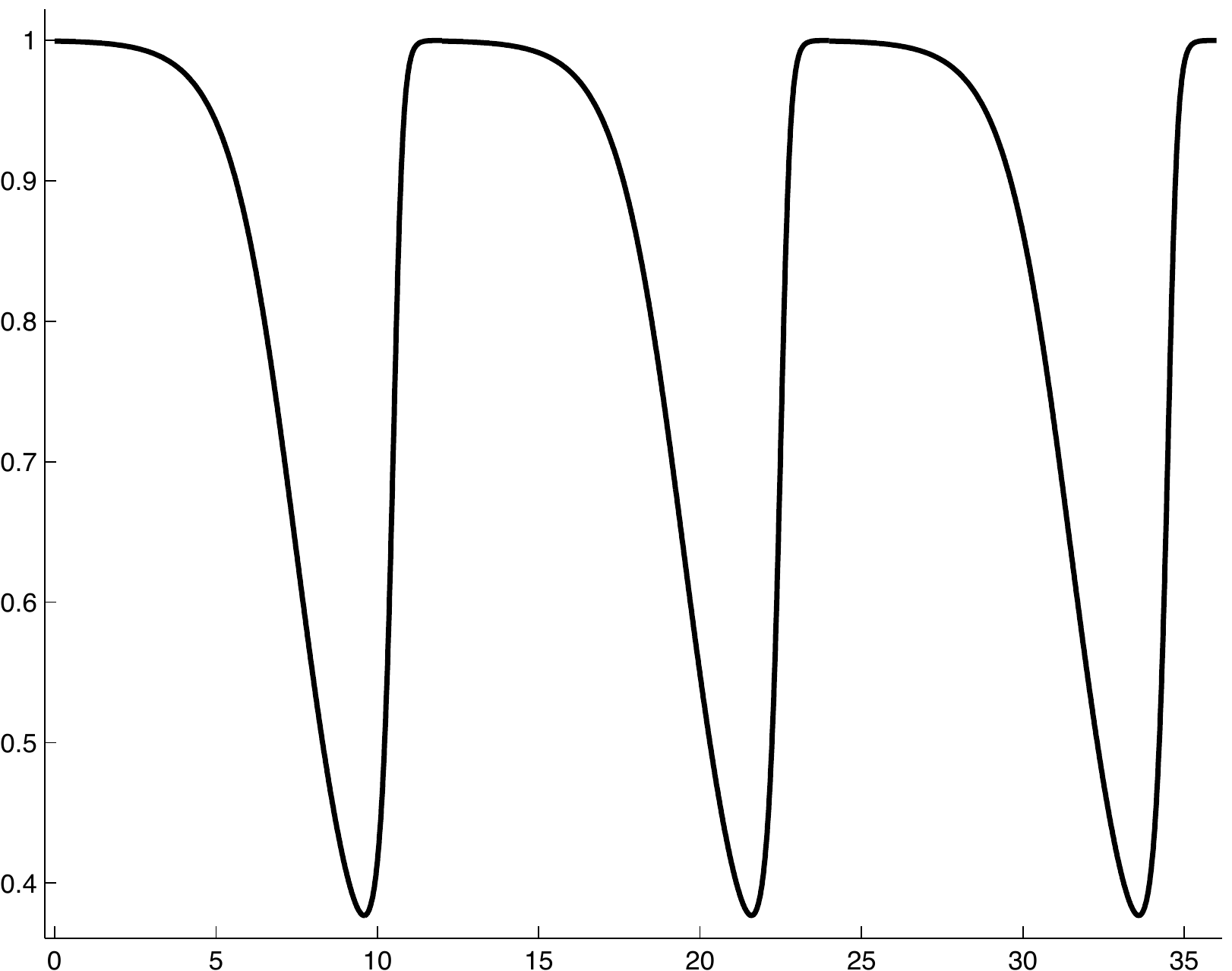}
 \end{array}
 $
\end{center}
  \caption{
In Figure (a), we plot the periodic spectrum of the periodically extended
version of the profile found in Section \ref{Case2} given in Figure (b).
This dynamic spectrum of the associated homoclinic profile is seen
to be stable, in contrast to its essential spectrum which was seen
to be unstable in Figure \ref{venant_evan_homoclinic}.
This spectrum was found using again the SpectrUW package developed at University of
Washington.  Note that the dynamic spectrum includes also small (stable) loops close to the point
spectrum of the homoclinic \cite{G}.
}
\label{dynamic_spectrum}
\end{figure}

This suggests also that there should exist nearby periodic wave trains
that are stable, despite the instability of a single solitary wave.
In companion papers \cite{BJNRZ1,BJNRZ2}, we show that this is indeed the case,
establishing spectral, linearized, and nonlinear stability of periodic
wave trains in a band near the homoclinic limit.
We point to the dynamic spectrum, or
similar quantification of de-amplification properties, as an
interesting direction for further development;
see \cite{BJNRZ1,BJNRZ2} for further discussion of this and related topics.
A further interesting direction for future study would be to
obtain pointwise Green function
bounds as in \cite{JZN,MaZ3}
for the metastable case, generalizing the bounds obtained in
\cite{OZ1} for perturbations of unstable constant solutions.
A related problem is to obtain rigorous nonlinear instability results
in the metastable case.
(Recall
that essential instabilitiy does not immediately imply nonlinear
instability due to the absence of a spectral gap.)
\appendix

\section{High Frequency Bounds}

Here, we record the quantities necessary to compute the high frequency bounds of Corollary~\ref{HFcorollary} in terms
of the original homoclinic orbit and of quantity $\bar\alpha$ defined in \eqref{alphaII}.
\begin{align*}
\Theta^0_{++}&=  \bp \frac{\bar\alpha\bar\tau^2}{c\nu}&-\frac{\bar\tau^2}{\nu}\\
-\frac{\bar\alpha\bar\tau^2}{2\nu}&\frac{\bar\tau'}{\bar\tau}-\frac{c\bar\tau^2}{2\nu}\ep,
\quad \Theta^1_{+-}= \bp
-\frac{2\bar\tau'}{\sqrt{\nu}}+\frac{\bar\tau^3}{\nu^{3/2}}(\frac{\bar\alpha}{c}+c)\\
-\frac{r\bar\tau^{s+2}\left(q-c\bar\tau\right)^{r-1}}{2\sqrt{\nu}}-\frac{\bar\alpha}{2}\frac{\bar\tau^3}{\nu^{3/2}}
\ep,
\\
\Theta^1_{++}&= \bp 0&
-\frac{2\bar\tau'}{\sqrt{\nu}}+\frac{\bar\tau^3}{\nu^{3/2}}(\frac{\bar\alpha}{c}+c)\\
\frac{\bar\tau}{2\sqrt{\nu}}\left((s+1)\bar\tau^s\left(q-c\bar\tau\right)^r+c\bar\alpha\frac{\bar\tau^2}{\nu}\right)&
\frac{r\bar\tau^{s+2}\left(q-c\bar\tau\right)^{r-1}}{2\sqrt{\nu}}+\frac{\bar\alpha}{2}\frac{\bar\tau^3}{\nu^{3/2}}
\ep,
\\
\Theta^2_{++}&=\bp -\frac{(s+1)\bar\tau^{s+2}\left(q-c\bar\tau\right)^r}{\nu}-\frac{c\bar\alpha}{\sqrt{\nu}}\frac{\bar\tau^3}{\nu^{3/2}}&
-\frac{r\bar\tau^{s+3}\left(q-c\bar\tau\right)^{r-1}}{\nu}\\
0&\frac{(s+1)\bar\tau^{s+2}\left(q-c\bar\tau\right)^r}{2\nu}+\frac{c\bar\alpha}{2\sqrt{\nu}}\frac{\bar\tau^3}{\nu^{3/2}}
\ep,
\\
\Theta^3_{++}&=\bp0&-\frac{(s+1)\bar\tau^{s+3}\left(q-c\bar\tau\right)^r}{\nu^{3/2}}-\frac{c\bar\alpha}{\sqrt{\nu}}\frac{\bar\tau^4}{\nu^2}\\
0&0
\ep,\quad
\Theta^0_{+-}= \bp \frac{\bar\tau^2}{\nu}\\
\frac{\bar\tau'}{2\bar\tau}-\frac{c}{2}\frac{\bar\tau^2}{\nu}\ep,
\\
\Theta^2_{+-}&= \bp
\frac{r\bar\tau^{s+3}\left(q-c\bar\tau\right)^{r-1}}{\nu}\\
\frac{(s+1)\bar\tau^{s+2}\left(q-c\bar\tau\right)^r}{2\nu}+\frac{c\bar\alpha}{2\sqrt{\nu}}\frac{\bar\tau^3}{\nu^{3/2}}
\ep,\quad
\Theta^3_{+-}= \bp
-\frac{(s+1)\bar\tau^{s+3}\left(q-c\bar\tau\right)^r}{\nu^{3/2}}-\frac{c\bar\alpha}{\sqrt{\nu}}\frac{\bar\tau^4}{\nu^2}\\
0
\ep,
\\
\Theta^0_{-+}&=  \bp \frac{\bar\alpha\bar\tau^2}{2\nu}&
\frac{\bar\tau'}{\bar\tau}-\frac{c\bar\tau^2}{2\nu}  \ep,\quad
\Theta^2_{-+}= \bp
0&\frac{(s+1)\bar\tau^{s+2}\left(q-c\bar\tau\right)^r}{2\nu}+\frac{c\bar\alpha}{2\sqrt{\nu}}\frac{\bar\tau^3}{\nu^{3/2}}
\ep,
\\
\Theta^1_{-+}&=  \bp
\frac{\bar\tau}{2\sqrt{\nu}}\left((s+1)\bar\tau^s\left(q-c\bar\tau\right)^r+c\bar\alpha\frac{\bar\tau^2}{\nu}\right)&
\frac{r\bar\tau^{s+2}\left(q-c\bar\tau\right)^{r-1}}{2\sqrt{\nu}}+\frac{\bar\alpha}{2}\frac{\bar\tau^3}{\nu^{3/2}}\ep,\quad
\Theta^0_{--}=\bp\frac{\bar\tau'}{\bar\tau}-\frac{c\bar\tau^2}{2\nu}\ep,
\\
\Theta^1_{--}&= \bp-\frac{(s+1)\bar\tau^{s+2}\left(q-c\bar\tau\right)^r}{2\nu}+\frac{c\bar\alpha}{2\sqrt{\nu}}\frac{\bar\tau^3}{\nu^{3/2}}\ep,\quad
\Theta^2_{--}=\bp\frac{(s+1)\bar\tau^{s+2}\left(q-c\bar\tau\right)^r}{2\nu}+\frac{c\bar\alpha}{2\sqrt{\nu}}\frac{\bar\tau^3}{\nu^{3/2}}\ep.
\end{align*}


\end{document}